\documentclass[11pt,a4paper]{article}
\usepackage{lmodern}
\usepackage[utf8]{inputenc}
\usepackage[english]{babel}
\usepackage[T1]{fontenc}
\usepackage{amsmath, amsfonts, amssymb, amsthm}
\usepackage{fullpage}
\usepackage{microtype}
\usepackage[shortlabels]{enumitem}
\setenumerate{label=$\arabic*)$}
\usepackage{mathtools} 
\usepackage{tikz}
\usetikzlibrary{decorations.markings}
\usepackage[smalltableaux]{ytableau}
\usepackage[colorlinks=true, linkcolor=blue]{hyperref}

\numberwithin{equation}{section}
\numberwithin{figure}{section}

\newtheorem{theorem}[equation]{Theorem}
\newtheorem{theoremintro}{Theorem}

\newtheorem{proposition}[equation]{Proposition}
\newtheorem{propositionintro}[theoremintro]{Proposition}

\newtheorem{corollary}[equation]{Corollary}

\newtheorem{lemma}[equation]{Lemma}

\theoremstyle{definition}
\newtheorem{definition}[equation]{Definition}

\theoremstyle{remark}
\newtheorem{remark}[equation]{Remark}
\newtheorem{example}[equation]{Example}

\newcommand\Set[1][e]{\widetilde{\mathfrak{S}}_{#1}}
\newcommand\order\unlhd
\newcommand\Y{\mathcal{Y}}
\newcommand\charge{\mathsf{c}}
\newcommand\emptypart[1][\charge]{\emptyset_{#1}}
\newcommand\mcharge{\underline{\charge}}
\newcommand\mpart[1][\lambda]{\underline{#1}}
\newcommand\mparti[2][\lambda]{#1^{(#2)}}
\newcommand\mempty{\mpart[\emptyset]}
\newcommand\setc{J_\charge}
\newcommand\level\ell
\newcommand\GA{\mathrm{GA}}
\newcommand\GAi[1][\charge]{\GA^{#1}}

\newcounter{pos}
\newcommand{\scaletikz}{.6}
\newcommand{\scalelozenge}{.7}

\newcommand\beadgen[2] 
{\fill[color=#1] (#2,0) circle (.21);}
\newcommand\bead[1]{\beadgen{red}{#1}}
\newcommand\bbead[1]{\beadgen{black}{#1}}
\newcommand\loz[1]{\draw[red] (#1,0) node[scale=\scalelozenge]{$\blacklozenge$};}

\newcommand\cun{blue}
\newcommand\ctrois{red}
\newcommand\cdeux{olive}
\newcommand\cquatre{black}

\newcommand\num[2]{
\draw (#1,-1) node[color=#2]{ $\scriptstyle #1$};}

\newcommand{\emptyabacusinside}[2]{
	\draw[very thick] (0,.5) -- (0,-.5);
	\foreach \i in {#1,...,#2}
					{
					\draw[very thin] (\i, .1) -- (\i, -.1);
					}
	\draw (#1 - .5, 0) node[left]{$\cdots$} -- (#2 + .5, 0) node[right]{$\cdots$};
	}

\newcommand{\abacusinside}[2]{
		\draw[very thick] (0, .5) -- (0, -.5);
		\setcounter{pos}{0}
		\foreach \b in {#2}
			{
			\ifnum \value{pos} = 0
				\draw (#1, 0) -- (#1+\b+2, 0) node[right] {$\cdots$};
				\foreach \i in {-1,...,\b}
					{
					\draw[very thin] (#1+\i+1, .1) -- (#1+\i+1, -.1);
					}
			\fi
			\stepcounter{pos}
			\fill (#1 + \b - \value{pos}, 0) circle (.2);
			}
			
		\foreach \i in {1,2}
			{
			\stepcounter{pos}
			\fill (#1 - \value{pos}, 0) circle (.2);
			}
			
		\foreach \i in {0,...,\value{pos}}
			{
			\draw[very thin] (#1-\i, .1) -- (#1-\i, -.1);
			}
		\stepcounter{pos}
		\draw (#1, 0) -- (#1-\value{pos}, 0) node[left]{$\cdots$};
}

\newcommand{\cabacus}[3]{
	\begin{center}
	\begin{tikzpicture}[scale=\scaletikz, baseline=(current bounding box.center)]
		\draw[very thick] (0, .5) -- (0, -.5);
		\setcounter{pos}{0}
		\foreach \b in {#2}
			{
			\ifnum \value{pos} = 0
				\draw (#1, 0) -- (#1+\b+2, 0) node[right] {$\cdots$};
				\foreach \i in {-1,...,\b}
					{
					\draw[very thin] (#1+\i+1, .1) -- (#1+\i+1, -.1);
					}
			\fi
			\stepcounter{pos}
			\fill (#1 + \b - \value{pos}, 0) circle (.2);
			}
			
		\foreach \i in {1,2}
			{
			\stepcounter{pos}
			\fill (#1 - \value{pos}, 0) circle (.2);
			}
			
		\foreach \i in {0,...,\value{pos}}
			{
			\draw[very thin] (#1-\i, .1) -- (#1-\i, -.1);
			}
		\stepcounter{pos}
		\draw (#1, 0) -- (#1-\value{pos}, 0) node[left]{$\cdots$};
		
	\end{tikzpicture} #3
	\end{center}
}

\newcommand{\ceabacus}[4]{
	\begin{center}
	\begin{tikzpicture}[scale=\scaletikz, baseline=(current bounding box.center)]
		\draw[very thick] (0, 0) -- (0, -#1 - 1);
		\setcounter{pos}{-#4}
		\foreach \b in {#2}
			{
			\stepcounter{pos}
			\ifnum \value{pos} = 1
				\pgfmathparse{(\b - \value{pos} - Mod(\b - \value{pos}, #1)) / #1}
				\foreach \i in {-1,...,-#1}
					{
					\draw (\pgfmathresult + 3, \i) -- (0, \i);
					\draw (\pgfmathresult + 3.6, \i - .03) node {$\dots$};
					\foreach \j in {-2,...,\pgfmathresult}
						\draw[very thin] (\j + 2, \i - .1) -- ++ (0, .2);
					}
			\fi
			\pgfmathparse{Mod(\b - \value{pos}, #1)}
			\fill (
					{(\b - \value{pos} - \pgfmathresult) / #1}
					,
					{- \pgfmathresult - 1}
					) circle (.2);
			}
			\foreach \i in {1,...,#1}
			{
			\stepcounter{pos}
			\pgfmathparse{Mod(-\value{pos}, #1) == #1 - 1}
			\ifnum \pgfmathresult = 0
				\pgfmathparse{Mod(-\value{pos}, #1)}
				\fill (
					{(-\value{pos} - \pgfmathresult) / #1}
					,
					{-\pgfmathresult - 1}
					) circle (.2);
			\else
				\breakforeach
			\fi
			}
			\foreach \k in {1,2}
			\foreach \i in {1,...,#1}
				{
				\pgfmathparse{Mod(-\value{pos}, #1)}
				\fill (
					{(-\value{pos} - \pgfmathresult) / #1}
					,
					{-\pgfmathresult - 1}
					) circle (.2);
				\stepcounter{pos}
				}

			\pgfmathparse{(-\value{pos} - Mod(-\value{pos}, #1)) / #1 + 1}
			\foreach \i in {-1,...,-#1}
				{
				\draw (\pgfmathresult - 1, \i) -- (0, \i);
				\draw (\pgfmathresult - 1.5, \i - .03) node {$\dots$};
			
				\foreach \j in {-\value{pos},...,0}
					{
					\ifnum \j > \pgfmathresult
						\draw[very thin] (\j, \i - .1) -- ++ (0, .2);
					\fi;
					}
				}
		
	\end{tikzpicture} #3
	\end{center}
}

\author{Salim \textsc{Rostam}\thanks{Institut Denis Poisson, CNRS UMR 7013, Université de Tours, 37200 Tours, France}}
\title{A window to the Bruhat order on the affine symmetric group}

\begin{document}

\maketitle

\abstract{Given two affine permutations, some results of Lascoux and Deodhar, and independently Jacon--Lecouvey, allow to decide if they are comparable for the strong Bruhat order. These permutations  are associated with tuples of core partitions, and the preceding problem is equivalent to compare the Young diagrams in each components for the inclusion. Using abaci, we give an easy rule to compute these Young diagrams one another. We deduce a procedure to compare, for the Bruhat order, two affine permutations in the window notation.}

\section{Introduction}

Given a Coxeter system $(W,S)$, the (strong) Bruhat order $\order$ is defined by $w \order w'$ if and only if a reduced expression of $w$ can be obtained as a subword of a reduced expression of $w'$. A finer version is given by the weak Bruhat order, where ``subword'' is replaced by ``suffix''. The notion of Bruhat order was first introduced by Ehresmann~\cite{ehresmann} in the context of Schubert varieties, and is still much studied (see, for instance, \cite{chapelier,elias,huang,nicolaides}).

In the case where $W$ is the symmetric group $\mathfrak{S}_e$ of index $e \geq 2$, it is natural to ask whether $\sigma,\rho \in \mathfrak{S}_e$ are comparable for the  Bruhat order from the data of $\sigma(1),\dots,\sigma(e)$ and $\rho(1),\dots,\rho(e)$, instead of dealing with reduced expressions. A celebrated criterion is the ``tableau criterion'' (see for instance~\cite{ehresmann,bjorner-brenti:tableau}) which uses the inclusion order on subsets of $\{1,\dots,e\}$ obtained by reordering the $j$-tuples $(\sigma(1),\dots,\sigma(j))$ for all $j \in \{1,\dots,e-1\}$. For the weak Bruhat order, such an interpretation uses inclusion of \emph{inversion sets}, see for instance~\cite[Proposition 2.1]{inv} or~\cite[Chapter 3]{bjorner-brenti:combinatorics}.

We now focus on the \emph{affine} symmetric group $\Set$. Restricting to some particular subsets of $\Set$ called \emph{affine Grassmannians}, Lascoux~\cite{lascoux} gave a criterion that translates the notion of Bruhat order into inclusion of Young diagrams, where the underlying partitions are \emph{$e$-cores}.  This notion of $e$-cores appears in the modular representation theory of the symmetric group~$\mathfrak{S}_e$ (see for instance~\cite{james-kerber}). Deodhar~\cite{deodhar} showed how to deduce the Bruhat order on $\Set$ from the Bruhat order on all affine Grassmannianns. Note that Jacon--Lecouvey~\cite[\textsection 5.4]{jacon-lecouvey:keys} obtained, using crystals theory, the same rule via reduction to the type A case.

Lascoux's construction uses reduced expressions, hence it is natural to ask for a simple criteria using the data of  $\sigma(1),\dots,\sigma(e)$ for $\sigma \in \Set$ (these values determine uniquely $\sigma \in \Set$), which is called the \emph{window notation}. Note that such a, quite complex, rule can be found in~\cite[\textsection 6]{bjorner-brenti:affine}. Concerning  the weak Bruhat order, there is a nice generalisation of the criteria for the weak Bruhat order in $\mathfrak{S}_e$, via inclusions of multigraphs (see~\cite[\textsection 5]{bjorner-brenti:affine}).

\medskip
In this paper, we will give a simple rule from the window notation to decide whether two affine permutations are comparable for the Bruhat order. 
 It will first consists in computing the $e$-core partition of Lascoux~\cite{lascoux}, then adding certain \emph{rim hooks} and finally deleting the first column. A key point is to use the representation on integer partitions via \emph{abaci}, where~$\Set$ acts in a very simple way. More precisely, if $w \in \Set$ is associated with the $e$-tuple of abaci $(A_0,\dots,A_{e-1})$ and the $e$-tuple of charged partitions $(\lambda^{(0)},\dots,\lambda^{(e-1)})$ then in Proposition~\ref{proposition:abacus_w} we first explicit the abacus $A_0$.
 
\begin{propositionintro}
\label{propositionintro}
We have $A_0 = \sqcup_{i = 0}^{e-1} \bigl(w(i) + e\mathbb{Z}_{< 0}\bigr)$.
\end{propositionintro}

Then in Proposition~\ref{proposition:procedure} we show how to construct each $A_{\charge}$ from $A_{\charge-1}$.

\begin{propositionintro}
Let $\charge \in \{1,\dots,e-1\}$. The abacus $A_\charge$ is obtained from $A_{\charge-1}$ by adding a bead at $w(\charge)$.
\end{propositionintro}

Now our main result Theorem~\ref{theorem:procedure_partitions} converts the above result in the context of partitions.

\begin{theoremintro}
\label{theoremintro}
For any $\charge \in \{1,\dots,e-1\}$, the $\charge$-charged $e$-core $\mparti{\charge}$ can be constructed from the $(\charge-1)$-charged $e$-core $\mparti{\charge-1}$ in the following way.
\begin{enumerate}
\item Construct the partition $\mu_{\charge-1,\charge}$ obtained from $\mparti{\charge-1}$ by adding the smallest rim hook with the following properties:
\begin{itemize}
\item it has only one node below the first column of the Young diagram of $\mparti{\charge-1}$,
\item it has hand reside $w(\charge)-1$.
\end{itemize}
\item Remove the first column of $\mu_{\charge-1,\charge}$.
\end{enumerate}
\end{theoremintro}
Now combining results of Lascoux~\cite{lascoux} and Deodhar~\cite{deodhar}, we know that two affine permutations $w,w' \in \Set$ satisfy $w \order w'$ if and only if the Young diagram of $\lambda^{(\charge)}$ is included in the Young diagram of $\lambda'^{(\charge)}$ for all $\charge \in \{0,\dots,e-1\}$, where $w$ (resp. $w'$) is associated to $(\lambda^{(0)},\dots,\lambda^{(e-1)})$ (resp. $(\lambda'^{(0)},\dots,\lambda'^{(e-1)})$). Hence, using Proposition~\ref{propositionintro} and Theorem~\ref{theoremintro}, given $w,w' \in \Set$ we have a procedure to decide whether $w \order w'$ from the knowledge of $(w(1),\dots,w(e))$ and $(w'(1),\dots,w'(e))$, that is, from the window notations of $w$ and $w'$.

\medskip

We now give the outline of the paper. In Section~\ref{section:partitions} we recall basic definitions concerning charged integer partitions, their abaci and the $e$-core property. Section~\ref{section:affine_sg}  recalls some known facts about the affine symmetric group $\Set$, for instance, its natural action on the set of $e$-cores. Namely, in Proposition~\ref{proposition:abacus_w} we explicit the abacus associated with any element $w \in \Set$. In~\textsection\ref{subsection:grassmannian} and~\textsection\ref{subsection:bruhat_grassmannian}, we study particular subsets of $\Set$, the \emph{affine Grassmannians}. The assertions there are meticulously proved, since they are used in Lascoux~\cite{lascoux} sometimes implicitly. For instance, in Proposition~\ref{proposition:addable_node_siw} we give the link between having addable $i$-nodes for the $e$-core associated with an affine Grassmannian element $w \in \Set$, and the length of $s_iw$. In Section~\ref{section:bruhat_whole_affine} we first invoke a result of Deodhar~\cite{deodhar} to obtain a criteria for the Bruhat order in the whole affine symmetric group $\Set$ (Corollary~\ref{corollary:lascoux_whole}). Then, in~\textsection\ref{subsection:inductive_abaci} we show how to use Lascoux's result to obtain an easy procedure using abaci (Proposition~\ref{proposition:procedure}) to decide whether two affine permutations are comparable for the Bruhat order from the window notation, and in~\textsection\ref{subsection:inductive_construction_partitions} we convert this procedure into the setting of Young diagrams (Theorem~\ref{theorem:procedure_partitions}). Finally, in Appendix~\ref{appendix_section:cores_higher_levels} we interpret our results in the context of $(e,\mcharge)$-core partitions of Jacon--Lecouvey~\cite{jacon-lecouvey:keys,jacon-lecouvey:cores}. Namely, in Proposition~\ref{proposition:length_addable} we give the link between having addable or removable $i$-nodes for the $(e,\mcharge)$-core associated with any $w \in \Set$, and the length of $s_iw$. Finally, in Proposition~\ref{proposition:lattices} we deduce that the lattice $\Set$ for the Bruhat order and the lattice of $(e,\charge)$-cores for the inclusions of multi-Young diagrams are isomorphic.

\paragraph{Acknowledgements} The author thanks Nicolas Jacon and Cédric Lecouvey for useful discussions. 
This research was funded, in whole or in part, by the
Agence Nationale de la Recherche funding ANR CORTIPOM 21-CE40-001. A CC-BY public
copyright license has been applied by the author to the present document and will be applied
to all subsequent versions up to the Author Accepted Manuscript arising from this submission,
in accordance with the grant's open access conditions.

\paragraph{Notation} In the whole paper we fix an integer $e \geq 2$.

\section{Partitions, abaci, cores}
\label{section:partitions}

We recall here some standard materials about partitions. We refer for instance to~\cite{james-kerber} for more details.

\begin{definition}
A \emph{partition} is a sequence $\lambda = (\lambda_1 \geq \dots \lambda_h > 0)$ with $h \geq 0$. A \emph{$\charge$-charged} partition is the data of partition and an integer $\charge \in \mathbb{Z}$.
\end{definition}

We now fix a charge $\charge \in \mathbb{Z}$. We will use the notation $\emptypart$ to denote the $\charge$-charged empty partition. 

The \emph{Young diagram} of a partition $\lambda = (\lambda_1,\dots,\lambda_h > 0)$ is:
\[
\Y(\lambda) \coloneqq \{(a,b) \in \mathbb{Z}^2 : 1 \leq a \leq h \text{ and } 1 \leq b \leq \lambda_a\}.
\]
A \emph{node} is an element of $(\mathbb{Z}_{\geq 1})^2$.
Young diagrams are usually represented with the $a$-axis (resp. $b$-axis) going down (resp. right).

\begin{example}
The Young diagram of the partition $\lambda = (3,2)$ is $\ydiagram{3,2}$, the top right node being $(1,3) \in \Y(\lambda)$.
\end{example}

If $\lambda$ is a partition and  $\gamma \in \mathbb{Z}^2$, we say that the node $\gamma$ is:
\begin{itemize}
\item \emph{addable} if $\gamma \notin \Y(\lambda)$ and $\Y(\lambda) \cup\{\gamma\}$ is the Young diagram of a partition,
\item \emph{removable} if $\gamma \in \Y(\lambda)$ and $\Y(\lambda) \setminus\{\gamma\}$ is the Young diagram of a partition.
\end{itemize}

\begin{lemma}
\label{lemma:rem_add_inodes_inclusion}
Let $\lambda,\mu$ be two partitions with $\Y(\lambda) \subseteq \Y(\mu)$.
\begin{itemize}
\item If $\gamma \in \Y(\lambda)$ is removable for $\mu$ then $\gamma$ is removable for $\lambda$.
\item If a node $\gamma$ is addable for $\lambda$ then either $\gamma \in \Y(\mu)$ or $\gamma \notin \Y(\mu)$ is addable for $\mu$.
\end{itemize}
\end{lemma}

\begin{proof}
\begin{itemize}
\item Let $\gamma = (a,b) \in \Y(\lambda)$ be removable for $\mu$. Then $(a+1,b) \notin \Y(\mu)$ and $(a,b+1)\notin \Y(\mu)$ thus $(a+1,b) \notin \Y(\mu)$ thus $(a,b+1) \notin \Y(\lambda)$ and $(a+1,b) \notin \Y(\lambda)$ thus $(a,b) = \gamma$ is removable for $\lambda$ as desired.
\item Let $\gamma = (a,b)$ addable for $\lambda$. If $\gamma \in \Y(\mu)$ we are done, thus now assume that $\gamma \notin \Y(\mu)$ and let us prove that $\gamma$ is addable for $\mu$. If $a,b \geq 2$ then $(a-1,b),(a,b-1) \in \Y(\lambda)$ thus $(a-1,b),(a,b-1)\in\Y(\mu)$ thus $(a,b) = \gamma$ is addable for $\mu$ as desired. The cases where $a = 1$ or $b = 1$ is similar, or can be deduced by including the nodes with a $0$-coordinate in the Young diagram.
\end{itemize}
\end{proof}

We now turn to a third representation for partitions.

\begin{definition}
\label{definition:abacus}
We say that $A \subseteq \mathbb{Z}$ is an \emph{abacus} if there exists $s \geq 0$ such that $\mathbb{Z}_{\leq -s} \subseteq A$ and $\mathbb{Z}_{\geq s} \cap A = \emptyset$.
\end{definition}

If $A$ is an abacus and $x \in \mathbb{Z}$, we say that $x$ is a \emph{bead} (resp. \emph{gap}) of $A$ if $x \in A$ (resp. $x \notin A$). Sliding all the beads of an abacus to the left (if $\mathbb{Z}_{\leq -s} \subseteq A$ then we only move the beads of $\mathbb{Z}_{> -s} \cap A$, which is finite) gives a subset of $\mathbb{Z}$ of the form $\mathbb{Z}_{\leq c}$. We say that $c$ is the \emph{charge} of $A$.

\begin{example}
\label{example:abacus}
We consider the following abacus, where the vertical line denotes the place $0$.
\begin{center}
\begin{tikzpicture}[scale=\scaletikz]
\emptyabacusinside{-4}{9}
\foreach \i in {-4,-3,-2,1,3,4,7}
	{\bbead{\i}}
\end{tikzpicture}
\end{center}
Sliding all the beads to the left gives:
\begin{center}
\begin{tikzpicture}[scale=\scaletikz]
\emptyabacusinside{-4}{9}
\foreach \i in {-4,...,2}
	{\bbead{\i}}
\end{tikzpicture}
\end{center}
thus the charge of the abacus is $2$.
\end{example}

Note that, as in Example~\ref{example:abacus_action}, in abaci we will sometimes write the corresponding integers below, with colors corresponding to an $e$-residue class.

\begin{definition}
Let $\lambda = (\lambda_1 \geq \dots \geq \lambda_h > 0)$ be a $\charge$-charged partition. The \emph{abacus} of $\lambda$ is $\{\lambda_k - k + \charge + 1 : k \geq 1\}$, with $\lambda_k \coloneqq 0$ if $k > h$.
\end{definition}

The abacus of a charged partition is an abacus indeed in the sense of Definition~\ref{definition:abacus}. Conversely, any abacus is the abacus of a charged partition, by counting the number of gaps to the left of each bead.

\begin{example}
The abacus of Example~\ref{example:abacus} is the abacus of the $2$-charged partition $(5,3,3,2)$.
\end{example}

\begin{definition}
Let $\lambda$ be a $\charge$-charged partition. The \emph{$e$-residue} (or simply \emph{residue}) of node $(a,b) \in \Y(\lambda)$ is $b-a+\charge \pmod{e}$.  An \emph{$i$-node} is a node of residue $i \in \mathbb{Z}/e\mathbb{Z}$.
\end{definition}

\begin{example}
For $e = 5$ and the $2$-charged  partition $\lambda = (5,2,1)$, in each node of $\Y(\lambda)$ we indicate the corresponding residue:
\[
\ytableaushort{23401,12,0}
\]
Note that for a $\charge$-charged partition, the node $(1,1)$ has always residue $\charge$.
\end{example}

The link between $i$-nodes and the residues of the beads in the abacus is as follows.

\begin{lemma}
\label{lemma:inodes_abacus}
Let $\lambda = (\lambda_1 \geq \dots \geq \lambda_h > 0)$ be a $\charge$-charged partition. Let $k \in \{1,\dots,h\}$ and let $x \coloneqq \lambda_k - k + \charge + 1$ be a bead of the abacus of $\lambda$. If $i \in \mathbb{Z}/e\mathbb{Z}$ is such that $x \equiv i+1\pmod{e}$ then the node $(k,\lambda_k) \in \Y(\lambda)$, which is at the end of the $k$-th row of $\Y(\lambda)$, is an $i$-node.
\end{lemma}

The fundamental result concerning abaci is the following (see, for instance, \cite[Lemma 2.7.13]{james-kerber}).

\begin{proposition}
\label{proposition:add_inode_abacus}
Let $\lambda$ be a charged partition and let $A$ be its abacus. Let $i \in \{0,\dots,e-1\}$. Then $\lambda$ has:
\begin{itemize}
\item an addable $i$-node if and only if there exists $x \in A$ with $x \equiv i \pmod{e}$ and $x+1 \notin A$,
\item a removable $i$-node if and only if there exists $y \in A$ with $y \equiv i+1 \pmod{e}$ and $y-1 \notin A$.
\end{itemize}
\end{proposition}

For instance, we obtain that the charge of the abacus of a $\charge$-charged partition is $\charge$ again (sliding all the beads to the left so that we obtain the abacus of $\emptypart$). 
We now turn to objects that will catch our interest.

\begin{definition}
\label{definition:core}
Let $\lambda$ be a charged partition and $A$ be its abacus. We say that $\lambda$ is an \emph{$e$-core} if for all $x \in A$ we have $x-e \in A$.
\end{definition}

\begin{remark}
By a generalisation of a result of Proposition~\ref{proposition:add_inode_abacus}, the property of a partition~$\lambda$ being an $e$-core is equivalent to the fact that the Young diagram of $\lambda$ has no $e$-hooks (see, for instance, \cite[2.7.40]{james-kerber}).  We will have no use of this property here, but note that we will use (rim) hooks in \textsection\ref{subsection:inductive_construction_partitions}.
\end{remark}

\begin{proposition}
\label{proposition:cores_AR}
Let $\lambda$ be an $e$-core and let $i \in \{0,\dots,e-1\}$. The $e$-core $\lambda$ cannot have both an addable and a removable $i$-node.
\end{proposition}

\begin{proof}
Let $A$ be an abacus associated with $\lambda$. By contradiction, if $x \in A$ (resp. $y \in A$) corresponds to an addable (resp. removable) $i$-node then:
\begin{itemize}
\item $x \equiv i \pmod{e}$ and $x+1 \in A$,
\item $y-1 \equiv i \pmod{e}$ and $y-1 \notin A$,
\end{itemize}
thus since $\lambda$ is an $e$-core we deduce that $x < y-1$. Similarly, we have:
\begin{itemize}
\item $x+1 \equiv i+1 \pmod{e}$ and $x+1 \notin A$,
\item $y \equiv i+1 \pmod{e}$ and $y \in A$,
\end{itemize}
thus $y < x+1$. This contradicts the previous inequality.
\end{proof}
%
%
%
%

\section{Affine symmetric group}
\label{section:affine_sg}

We give here mainly some known results about the affine symmetric group $\Set$, together with its action on the set of $e$-core partitions.

\subsection{Definition}
 We denote by $\Set$ the affine symmetric group, with generators $s_0,\dots,s_{e-1}$. The defining relations are:
\begin{align*}
s_i^2 &= 1, &&\text{for all } i \in \{0,\dots,e-1\},
\\
s_i s_j &= s_j s_i, &&\text{for all } i,j \in \{0,\dots,e-1\} \text{ with } |i-j| \geq 2,
\\
s_i s_{i+1} s_i &= s_{i+1}s_i s_{i+1}, &&\text{ for all } i \in \{0,\dots,e-1\},
\end{align*} 
where, in the last relation, the indices are taken modulo $e$.   Elements of $\Set$ can be identified with the set of permutations $w$ of $\mathbb{Z}$ such that:
\begin{equation}
\label{equation:w(i+e)}
w(i+e) = w(i)+e, \qquad \text{for all } i \in \mathbb{Z}.
\end{equation}
 and $w(1) + \dots + w(e) = \frac{e(e+1)}{2}$ (see, for instance, \cite[Proposition 3.3]{bjorner-brenti:affine}). Under this identification, the generator $s_i$ is given by permuting $i+ke$ and $i+1+ke$ for all $k \in \mathbb{Z}$. In the sequel, we will not make any difference between an element of $\Set$ an its corresponding permutation of~$\mathbb{Z}$.
 
\begin{lemma}
\label{lemma:w(i)_congr}
If $w \in \Set$ then $w(1),\dots,w(e)$ have all different residues modulo $e$.
\end{lemma}

\begin{proof}
We have $\mathbb{Z} = \sqcup_{i = 1}^e (i + e\mathbb{Z})$ thus, by~\eqref{equation:w(i+e)} and since $w$ is a bijection we have $\mathbb{Z} = \sqcup_{i = 1}^e \bigl(w(i)+e\mathbb{Z}\bigr)$.
\end{proof}

Equation~\eqref{equation:w(i+e)} implies that an element $w \in \Set$ is determined by the image of $\{1,\dots,e\}$. 

\begin{definition}
The \emph{window notation} of $w\in\Set$ is $[w(1),\dots,w(e)]$.
\end{definition}

The window notation has the advantage over the expression as a word in the generators $s_0,\dots,s_{e-1}$ that it is always the data of $e$ integers, while the length of a word can be arbitrarily large.
The aim of this paper is to decide whether two elements of $\Set$ in the window notation are comparable for the Bruhat order, without using the decomposition as a product of generators $s_0,\dots,s_{e-1}$. 

To conclude this part, let us recall the following result, where $\ell : \Set \to \mathbb{Z}_{\geq 0}$ denotes the Coxeter length function.

\begin{lemma}[\protect{\cite[Proposition 3.2]{bjorner-brenti:affine}}]
\label{lemma:length(ws_i)}
Let $w \in \Set$ and let $i \in \{0,\dots,e-1\}$. We have:
\[
\ell(ws_i) < \ell(w) \iff w(i) > w(i+1).
\]
\end{lemma}

\subsection{Action on \texorpdfstring{$e$}{e}-cores}
\label{subsection:action_ecores}

Throughout this part we fix a charge $\charge \in \mathbb{Z}$. 
The action of $\Set$ on $\mathbb{Z}$ induces an action on the subsets of $\mathbb{Z}$. We can easily check that the set of abaci of charge $\charge$ is stable under this action, hence it gives rise to an action on the set of $\charge$-charged partitions. In fact we can see from Definition~\ref{definition:core} that the set of $e$-cores is stable under this action, and by Proposition~\ref{proposition:add_inode_abacus} the generator $s_i$ adds or remove all $i$-nodes (recalling from Proposition~\ref{proposition:cores_AR} that $\lambda$ has either no addable $i$-nodes or no removable $i$-nodes). Hence, the group $\Set$ acts on the set of $\charge$-charged $e$-cores. We will often say that the $\charge$-charged $e$-core $\lambda$ is  \emph{associated} with $w \in \Set$ if $\lambda = w \cdot \emptypart$.

\begin{example}
\label{example:action_si}
Take $e = 4$ and $w = s_0s_1 s_2 s_1 s_3 s_0 s_2 s_1$. With $\charge = 1$, representing the residues in the boxes we have successively:
\[
\emptypart[1]
\xrightarrow{s_1}
\ytableaushort{1}
\xrightarrow{s_2} \ytableaushort{12}
\xrightarrow{s_0} \ytableaushort{12,0}
\xrightarrow{s_3} \ytableaushort{123,0,3}
\xrightarrow{s_1} \ytableaushort{123,01,3}
\xrightarrow{s_2} \ytableaushort{123,012,3,2}
\xrightarrow{s_1} \ytableaushort{123,012,3,2,1}
\xrightarrow{s_0}
\ytableaushort{1230,012,30,2,1,0}
\]
so that $w \cdot \emptypart[1] = (4,3,2,1,1,1)$.
\end{example}

We now give some lemmas about inclusions of Young diagrams of $e$-cores, in the continuation of Lemma~\ref{lemma:rem_add_inodes_inclusion}.

\begin{lemma}
\label{lemma:action_si_inclusion}
Let $i \in \{0,\dots,e-1\}$. If $\lambda$ and $\mu$ are two $e$-cores with either both no addable $i$-nodes or both no removable $i$-nodes then:
\[
\Y(\lambda) \subseteq \Y(\mu) \iff \Y(s_i\lambda) \subseteq \Y(s_i \mu).
\]
\end{lemma}

\begin{proof}
\begin{itemize}
\item Assume first that $\lambda,\mu$ has no addable $i$-nodes.
 Suppose that $\Y(\lambda) \subseteq \Y(\mu)$ and let us prove that $\Y(s_i\lambda) \subseteq \Y(s_i\mu)$. It suffices to prove that if $\lambda \in \Y(\lambda)$ is an $i$-node removable for $\mu$  then $\lambda$ is also removable for $\Y(\mu)$, but this follows immediately from Lemma~\ref{lemma:rem_add_inodes_inclusion}.
\item Now assume that $\lambda,\mu$ has no removable $i$-nodes.  Suppose that $\Y(\lambda) \subseteq \Y(\mu)$ and let us prove that $\Y(s_i\lambda) \subseteq \Y(s_i\mu)$. It suffices to prove that if $\gamma$ is an addable $i$-node for $\lambda$ then either $\gamma \in \Y(\mu)$ either $\gamma$ is addable for $\mu$, but again this follows immediately from Lemma~\ref{lemma:rem_add_inodes_inclusion}.
\end{itemize}
The above proves the necessary condition.  The converse implication follows from the fact that an $e$-core $\nu$ has no addable (resp. removable) $i$-nodes if and only if $s_i\nu$ has no removable (resp. addable) $i$-nodes.
\end{proof}

\begin{remark}
Without the assumption the result is false. For instance, if $\lambda$ has some addable $i$-nodes then with $\mu \coloneqq s_i\lambda$ we have $s_i \mu = \lambda$ thus $\Y(\lambda) \varsubsetneq \Y(\mu)$ but $\Y(s_i\lambda) = \Y(\mu) \nsubseteq \Y(\lambda) = \Y(s_i\mu)$.
\end{remark}

\begin{lemma}
\label{lemma:mixed}
Let $i \in \{0,\dots,e-1\}$. If $\lambda$ and $\mu$ are two $e$-cores with $\mu$ having no addable $i$-nodes then:
\[
\Y(\lambda) \subseteq \Y(\mu) \iff \Y(s_i\lambda) \subseteq \Y(\mu).
\]
\end{lemma}

\begin{proof}
If $s_i \lambda = \lambda$ then the result is clear. If $\lambda$ has some addable $i$-nodes then $\Y(\lambda) \subseteq \Y(s_i\lambda)$. Hence, the sufficient condition is clear, thus assume $\Y(\lambda) \subseteq \Y(\mu)$ and let us prove that $\Y(s_i\lambda) \subseteq \Y(\mu)$. If $\gamma $ is a node inside $\Y(s_i\lambda) \setminus \Y(\lambda)$ then $\gamma$ is an $i$-node and $\gamma$ is an addable node for $\lambda$. By Lemma~\ref{lemma:rem_add_inodes_inclusion}, since $\mu$ has no addable $i$-nodes we have $\gamma \in \Y(\mu)$ as desired.

Now if $\lambda$ has some removable $i$-nodes then $s_i\lambda$ has some addable $i$-nodes and we use the equivalence that we have just proved, using in addition that $s_i(s_i\lambda) = \lambda$.
\end{proof}

The next result allows a computation of the $e$-core associated with an affine permutation $w \in \Set$ expressed in the window notation.

\begin{proposition}
\label{proposition:abacus_w}
Let $w \in \Set$. The abacus of the $\charge$-charged $e$-core $w\cdot\emptypart$ is:
\[
\bigsqcup_{i = 1}^e \bigl(w(\charge+i) + e\mathbb{Z}_{< 0}\bigr).
\]
\end{proposition}

\begin{proof}

The abacus of the empty $\charge$-charged partition is $A \coloneqq \mathbb{Z}_{\leq \charge} = \charge + \mathbb{Z}_{\leq 0}$. We can write  $A = A_1 \sqcup \dots \sqcup A_e$ where $A_i =  \charge+i + e\mathbb{Z}_{< 0}$ for any $i \in \{1,\dots,e\}$. Since $w \in \Set$, for any $i \in \{1,\dots,e\}$ we have:
\[
w(A_i) = w\left( \charge+i + e\mathbb{Z}_{< 0}\right) = w( \charge+i) + e\mathbb{Z}_{< 0}.
\]
The abacus of  $w\cdot \emptypart$ is $w(A) = \cup_{i = 1}^e w(A_i)$, but now $w : \mathbb{Z} \to \mathbb{Z}$ is a bijection thus the union is disjoint.
\end{proof}

\begin{example}
\label{example:abacus_action}
Let $e = 4$ and $\charge = 1$. We consider $w \in \Set$ given by $[5,-4,2,7]$ in the window notation.  We have:
\begin{align*}
w(\charge+1) &= w(2) = -4, & w(\charge+2) &= w(3) = 2,
\\
w(\charge+3) &= w(4) = 7, & w(\charge+4) &= w(\charge)+4 = w(1)+4 = 9,
\end{align*}
so that by Proposition~\ref{proposition:abacus_w} we know that the abacus of $w\cdot \emptypart$ is, where the gaps given by the above values are represented as red diamonds:
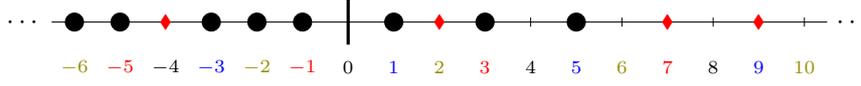
\begin{center}
\begin{tikzpicture}[scale=\scaletikz]
\emptyabacusinside{-6}{10}
\loz{-4}

\loz{2}
\bbead{-2}
\bbead{-6}

\loz{7}
\bbead{3}
\bbead{-1}
\bbead{-5}

\loz{9}
\bbead{5}
\bbead{1}
\bbead{-3}

\foreach \i in {-4,0,4,8} {\num{\i}{\cquatre}}
\foreach\i in {-5,-1,3,7}{\num{\i}{\ctrois}}
\foreach \i in {-6,-2,2,6,10}{\num{\i}{\cdeux}}
\foreach \i in {-3,1,5,9}{\num{\i}{\cun}}

\end{tikzpicture}
\end{center}
The associated partition is $(4,3,2,1,1,1)$. In fact this element $w$ is the same as the one of Example~\ref{example:action_si}, and we see that the calculation of $w\cdot\emptypart$ is here much shorter.
\end{example}

\subsection{Grassmannian elements}
\label{subsection:grassmannian}

Here we fix a charge $\charge \in \{0,\dots,e-1\}$. 
The assertions in this part can be found in Lascoux~\cite{lascoux} and are sometimes implicit. We provide some details for the proofs (see also~\cite[\textsection 3.2]{berg-jones-vazirani}).

\begin{definition}
\label{definition:GA}
We define $\GAi$ to  be the subset of $\Set$ formed by the elements $w \in \Set$ such that $\ell(ws_i) > \ell(w)$ for all $i \in \setc \coloneqq \{0,\dots,e-1\}\setminus\{\charge\}$.
\end{definition}

We can give an easy characterisation of the elements of $\GAi$ in terms of reduced expression.

\begin{lemma}
\label{lemma:GA_reduced_expression}
Let $w \in \Set$. We have  $w \in \GAi$ if and only if $w = 1$ or any reduced expression of $w$ ends by $s_\charge$.
 \end{lemma} 
 
\begin{proof}
The case $w = 1$ is clear, thus we suppose $w \neq 1$.  Assume that there exists a reduced expression $w = s_{i_1}\cdots s_{i_n}$ with $i_n \neq \charge$. Then $i_n \notin \setc$ and $ws_{i_n} = s_{i_1}\cdots s_{i_{n-1}}$ satisfies $\ell(ws_{i_n}) < \ell(w)$ thus $w \notin \GAi$. Conversely, assume that $w \notin \GAi$ so that there exists $i \in \setc$ with $\ell(ws_i) < \ell(w)$. Then if $w s_i = s_{i_1}\cdots s_{i_{n-1}}$ is a reduced expression then $w = s_{i_1}\cdots s_{i_{n-1}} s_i$ is a reduced expression with $i \neq \charge$.
\end{proof}

 The next result appears for instance in~\cite[Lemma 3.1]{deodhar}).

\begin{lemma}
\label{lemma:siw_GA}
Let $w \in \GAi$ and $i \in \setc$. If $\ell(s_i w) < \ell(w)$ then $s_i w \in \GAi$.
\end{lemma}

\begin{proof}
Suppose that $\ell(s_i w) < \ell(w)$. If $s_{i_1}\cdots s_{i_n}$ is a reduced expression of $s_i w$ then $s_i s_{i_1}\cdots s_{i_n}$ is a reduced expression of $w \in \GAi$ thus $i_n = \charge$ by Lemma~\ref{lemma:GA_reduced_expression}. By Lemma~\ref{lemma:GA_reduced_expression} again we deduce that $s_iw \in \GAi$.
\end{proof}

Definition~\ref{definition:GA} is in fact a particular case of a  standard result in the theory of Coxeter groups (see, for instance,  \cite[\textsection 1.10]{humphreys}).

\begin{proposition}
\label{proposition:representative}
Let $J \subseteq \{0,\dots,e-1\}$, define $(\Set)^J \coloneqq \{w \in \Set : \ell(ws) > \ell(w)$ for all $s \in J\}$ and consider the parabolic subgroup $(\Set)_J$ generated by the elements $s_i$ for $i \in J$. Then for any $w \in \Set$, there exists a unique $(w^J,w_J) \in (\Set)^J \times (\Set)_J$ such that $w = w^J w_J$. Moreover this unique couple satisfies $\ell(w) = \ell(w^J) + \ell(w_J)$.
\end{proposition}

The elements of the set $(\Set)^J$ as in Proposition~\ref{proposition:representative} are called \emph{minimal length coset representatives}. Recall from Definition~\ref{definition:GA} that $\setc = \{0,\dots,e-1\} \setminus\{\charge\}$.

\begin{definition}
\label{definion:w^c}
For $w \in \Set$, we write $w^{(c)} \in \GAi$ instead of the element $w^{\setc}$ of Proposition~\ref{proposition:representative}.
\end{definition}

The next result connects the projection on the affine Grassmannian with the action on charged $e$-cores, and will be central in Section~\ref{section:bruhat_whole_affine}.

\begin{proposition}
\label{proposition:action_w_wcharge}
Let $w \in \Set$. The elements $w$ and $w^{(\charge)}$ have the same action on $\emptypart$.
\end{proposition}

\begin{proof}
By Proposition~\ref{proposition:representative}, we can write $w = w^{(\charge)} s_{i_1} \cdots s_{i_n}$ with $i_j \in \setc$ for all $j \in \{1,\dots,n\}$. The only addable box in $\emptypart$ has residue $\charge$, thus $s_i \cdot \emptypart = \emptypart$ for all $i \in \setc$. We deduce that $w \cdot \emptypart = w^{(\charge)} \cdot \emptypart$ and this concludes the proof.
\end{proof}

We now give the analogue of Lemma~\ref{lemma:GA_reduced_expression} in terms of window notation (see~\cite[Proposition 3.5]{bjorner-brenti:affine}).

\begin{proposition}
\label{proposition:charact_GA}
For $w \in \Set$ we have $w \in \GAi \iff w(\charge+1) < \dots < w(\charge+e)$.
\end{proposition}

\begin{proof}
Let $w \in \Set$. By Lemma~\ref{lemma:length(ws_i)} we have:
\begin{align*}
w \in \GA &\iff \ell(ws_i) > \ell(w) &&\text{for all } i \in \setc
\\
&\iff w(i) < w(i+1) &&\text{for all } i \in \setc
\\
&\iff w(0) < \dots < w(\charge) \text{ and } w(\charge+1) < \dots < w(e).
\end{align*}
By~\eqref{equation:w(i+e)}, we have $w(0) < \dots < w(\charge) \iff w(e) < \dots < w(\charge+e)$ and we obtain the desired result.
\end{proof}

\begin{example}
The element $w$ of Example~\ref{example:abacus_action} belongs to $\GAi[1]$. Note that the same element~$w$ was considered in Example~\ref{example:action_si}, however there it was not obvious that $w \in \GAi[1]$.
\end{example}

\begin{remark}
For $w \in \Set$, the element $w^{(\charge)} \in \GAi$ of Definition~\ref{definion:w^c} is the unique affine permutation $w' \in \Set$ such that $\{w'(\charge+1) < \dots < w'(\charge+e)\} = \{w(\charge+1),\dots,w(\charge+e)\}$.
 Indeed, 
let $\sigma$ be the permutation of $\{\charge+1,\dots,\charge+e\}$ such that $w\sigma(\charge+1) < \dots < w\sigma(\charge+e)$. Then $\sigma$ is the product of some transpositions $(k,k+1)$ with $k \in \{\charge+1,\dots,\charge+e-1\}$, thus if $\sigma' \in \Set$ is given by this same product but replacing $(k,k+1)$ by the generator $s_{k\bmod e} \in \Set$ then:
\begin{itemize}
\item we have $\sigma' \in (\Set)_{\setc}$ since the generator $s_\charge$ does not appear in the above decomposition of $\sigma'$ into generators,
\item the element $w' \coloneqq w\sigma'$ satisfies $w'(\charge+1)<\dots<w'(\charge+e)$ thus $w' \in \GAi$ by Proposition~\ref{proposition:charact_GA}.
\end{itemize}
Hence, by unicity in Proposition~\ref{proposition:representative}, since $w = w' \sigma'^{-1}$ we obtain that $w' = w^{(\charge)}$, as desired.
\end{remark}

The next result allows to detect when the $e$-core associated with $w \in \GAi$ has an addable $i$-node.

\begin{proposition}
\label{proposition:addable_node_siw}
Let $w \in \GAi$ and let $\lambda = w \cdot \emptypart$ be the associated $\charge$-charged $e$-core. Let $i \in \{0,\dots,e-1\}$. Then $\lambda$ has an addable $i$-node if and only if $\ell(s_iw) > \ell(w)$ and $s_i w \in \GAi$.
\end{proposition}

\begin{proof}
First assume that $\lambda$ has an addable $i$-node. By Propositions~\ref{proposition:add_inode_abacus} and~\ref{proposition:abacus_w}, we can find $x,y \in \{\charge+1,\dots,\charge+e\}$ and $\delta \geq 1$ such that the abacus of $\lambda$ has a bead at position  $w(y) - \delta e$, a gap at position $w(x) = w(y)-\delta e+1$ and:
\begin{equation}
\label{equation:proof_GA_addable_residue}
w(y) \equiv i\pmod{e} \text{ and } w(x) \equiv i +1\pmod{e}.
\end{equation}
Note that since $\delta \geq 1$ and $e \geq 2$, we have $w(x) < w(y)$. 
If $k \in \mathbb{Z}$ is such that $w(x) = i+1+ke$ then by~\eqref{equation:w(i+e)} we have $w^{-1}(i+1) = x-ke$ and similarly $w^{-1}(i) = y - (k+\delta)e$. Now $\delta,e \geq 1$ and $|x-y| < e$ thus $w^{-1}(i) < w^{-1}(i+1)$, thus by Lemma~\ref{lemma:length(ws_i)} we have $\ell(w^{-1}s_i) > \ell(w^{-1})$ thus  $\ell(s_iw) > \ell(w)$.

We now have to  prove that $s_i w \in \GAi$. By~\eqref{equation:proof_GA_addable_residue} we have $s_iw(x) = w(x)-1$ and $s_i w(y) = w(y) + 1$. Recalling that $w(x) < w(y)$, we deduce that:
\begin{itemize}
\item by Proposition~\ref{proposition:charact_GA} we have $x <y$, since $w \in \GAi$ and $x,y \in \{\charge+1,\dots,\charge+e\}$,
\item we have $s_iw(x) = w(x)-1 < w(y)+1 = s_iw(y)$.
\end{itemize}
Now, for any $z \in \{\charge+1,\dots,\charge+e\}\setminus\{x,y\}$ we have $s_iw(z) = w(z)$, thus by~\eqref{equation:proof_GA_addable_residue} and Lemma~\ref{lemma:w(i)_congr} we have $w(z) < w(x) \iff w(z) < w(x)-1 \iff s_iw(z) < s_iw(x)$. Similarly we have $w(z) < w(y) \iff s_iw(z) < s_iw(y)$ and thus, by another application of Proposition~\ref{proposition:charact_GA}, we deduce that $s_i w \in \GAi$.

\medskip
Conversely, assume that $\ell(s_i w) > \ell(w)$ and $s_i w \in \GAi$ and let us prove that $\lambda$ has an addable $i$-node. By Lemma~\ref{lemma:length(ws_i)} we have $w^{-1}(i) < w^{-1}(i+1)$ and let $x,y \in \{1,\dots,e\}$ and $k,\delta \in \mathbb{Z}$ such that $w^{-1}(i+1) = x-ke$ and $w^{-1}(i) = y-(k+\delta)e$. Since $w^{-1}(i) < w^{-1}(i+1)$ we have:
\begin{equation}
\label{equation:proof_addable_node_siw_ba}
y - \delta e < x.
\end{equation}
Since $x,y \in \{\charge+1,\dots,\charge+e\}$, we deduce that $\delta \geq 0$. By~\eqref{equation:w(i+e)}, we have:
\begin{equation}
\label{equation:proof_addable_node_siw_w(a)_w(b)}
w(x) = i+1+ke \text{ and }w(y) = i+(k+\delta)e = w(x) + \delta e - 1,
\end{equation}
in particular $s_iw(x) = w(x)-1$ and $s_iw(y) = w(y)+1$ thus:
\[
s_iw(x) = i+ke \text{ and } s_iw(y) = i+1+(k+\delta)e = s_iw(x) + 1+\delta e.
\]
Since $\delta \geq 0$ we deduce that $s_iw(x) < s_iw(y)$ thus since $x,y \in \{\charge+1,\dots,\charge+e\}$ and $s_i w \in \GAi$ we obtain that $x < y$. Thus by~\eqref{equation:proof_addable_node_siw_ba} we have $\delta \neq 0$, thus $\delta \geq 1$. Now by~\eqref{equation:proof_addable_node_siw_w(a)_w(b)} we have
 $w(x) \leq w(y) - e+ 1$. By Proposition~\ref{proposition:abacus_w}, we deduce that the abacus of $\lambda$ has:
 \begin{itemize}
 \item a bead at position $w(y) - e \equiv i \pmod{e}$ and at gap at $w(y)$,
 \item a gap at position $w(x) \leq w(y) - e + 1$ with $w(x) \equiv i+1 \pmod{e}$.
 \end{itemize}
Since $\lambda$ is an $e$-core, by Definition~\ref{definition:core} we deduce that there is a gap at position $w(y)-e+1 \equiv i+1 \pmod{e}$, thus the bead at $w(y)-e$ has a gap directly to its right.  By  Proposition~\ref{proposition:add_inode_abacus}, this proves that $\lambda$ has an addable $i$-node.
\end{proof}

\begin{remark}
If, given $w \in \GAi$, we only assume that $\ell(s_i w) > \ell(w)$ then it does not necessarily mean that $\lambda$ has an  addable $i$-node.  For instance,  with $e = 4$ and  $w = s_0 \in \GAi[0]$ we have $\lambda = (1)$, here $\ell(s_2w) = 2 > \ell(w)$ but $\lambda$ does not have any addable $2$-node (the only addable nodes are $0$- and $3$-nodes). We will give in Proposition~\ref{proposition:length_addable} a statement without the Grassmannian assumption, using generalised $e$-cores.
\end{remark}

\begin{corollary}
\label{corollary:GA_cores_bij}
The map $\GAi \to \{\charge$-charged $e$-cores$\}$ given by $w \mapsto w\cdot\emptypart$ is a bijection.
\end{corollary}

\begin{proof}
Let us first prove that the map is onto. Let $\lambda$ be a $\charge$-charged $e$-core. We proceed by induction on the number of boxes of $\lambda$. If $\lambda = \emptypart$ then $\lambda = 1 \cdot \emptypart$ and this concludes the initialisation since $1 \in \GAi$. Now if $\lambda$ has at least one box, then $\lambda$ has at least a removable node. This removable node is a removable $i$-node for some $i \in \{0,\dots,e-1\}$. Thus, the partition $\lambda' \coloneqq s_i \lambda$ has strictly less boxes than $\lambda$ and $\lambda'$ has some addable $i$-nodes. By induction, we know that there exists $w' \in \GAi$ such that $\lambda' = w'\cdot\emptypart$. Now by Proposition~\ref{proposition:addable_node_siw}, we know that $w \coloneqq s_i w' \in \GAi$ and this concludes the proof since $w\cdot\emptypart = s_iw' \cdot \emptypart = s_i\lambda' = \lambda$.

To prove the injectivity, let $w,w' \in \GAi$ such that $w\cdot \lambda = w'\cdot \lambda$. By Proposition~\ref{proposition:abacus_w} we thus have:
\[
\bigsqcup_{i = 1}^e \bigl(w(\charge+i) + e\mathbb{Z}_{< 0}\bigr) = \bigsqcup_{i = 1}^e \bigl(w'(\charge+i) + e\mathbb{Z}_{< 0}\bigr).
\]
Hence, we have $\{w(\charge+1),\dots,w(\charge+e)\} = \{w'(\charge+1),\dots,w'(\charge+e)\}$, but now $w,w' \in \GAi$ thus $w(\charge+1) < \dots < w(\charge+e)$ and $w'(\charge+1) < \dots < w'(\charge+e)$ by Proposition~\ref{proposition:charact_GA} thus $w(\charge+i) = w'(\charge+i)$ for all $i \in \{1,\dots,e\}$ thus $w = w'$ (by~\eqref{equation:w(i+e)}).
\end{proof}

\begin{remark}
The proof of the injectivity in Corollary~\ref{corollary:GA_cores_bij} provides an easy way to recover $w \in \GAi$ from the abacus of a given $\charge$-charged $e$-core.
\end{remark}

The above two results now allow to translate the  meaning of having a removable $i$-node into the Coxeter setting, which is the key point of~\cite[Proposition 4.1]{lascoux}.

\begin{proposition}
\label{proposition:removable_node_siw}
Let $w \in \GAi$ and let $\lambda = w \cdot \emptypart$ be the associated $\charge$-charged $e$-core. Let $i \in \{0,\dots,e-1\}$. Then $\lambda$ has a removable $i$-node if and only if $\ell(s_iw) < \ell(w)$.
\end{proposition}

\begin{proof}
Assume that $\lambda$ has a removable $i$-node. Thus the $e$-core $\lambda' \coloneqq s_i\lambda$ has an addable $i$-node. By Corollary~\ref{corollary:GA_cores_bij}, there exists $w' \in \GAi$ such that $\lambda' = w'\cdot\emptypart$. We can thus apply Proposition~\ref{proposition:addable_node_siw} to find that:
\begin{equation}
\label{equation:proof_irem_ell}
\ell(s_iw') > \ell(w'),
\end{equation} and $s_iw' \in \GAi$. We now have $s_iw' \cdot \emptypart = s_i \lambda' = \lambda = w\cdot\emptypart$ and $w,s_iw' \in \GAi$ thus $w = s_iw'$ by Corollary~\ref{corollary:GA_cores_bij}. Thus~\eqref{equation:proof_irem_ell} gives $\ell(w) > \ell(s_iw)$ as desired.

We now assume that $\ell(s_iw) < \ell(w)$. By Lemma~\ref{lemma:siw_GA} we have $w' \coloneqq s_i w \in \GAi$. With $\lambda' \coloneqq w' \cdot\emptypart$, we thus have $\ell(s_iw') = \ell(w) > \ell(s_iw) = \ell(w')$ and $s_iw' = w \in \GAi$. Thus, by Proposition~\ref{proposition:addable_node_siw} the $e$-core $\lambda'$ has an addable $i$-node, thus $\lambda = w\cdot\emptypart = s_i w' \cdot\emptypart = s_i \lambda'$ has a removable $i$-node.
\end{proof}

\begin{remark}
If we do not assume that $w \in \GAi$ then we can have $\ell(s_i w) < \ell(w)$ without $\lambda$ having a removable $i$-node. For instance, with $e = 4$ and $w = s_1 s_2 \notin \GAi[0]$, we have $\lambda = \emptypart[0]$ thus $\lambda$ has no removable node although $\ell(s_1 w) = 1 < \ell(w)$. (See Proposition~\ref{proposition:length_addable}. for a statement without the Grassmannian assumption.)
\end{remark}

\subsection{Strong Bruhat order on the affine Grassmannians}
\label{subsection:bruhat_grassmannian}

In this part we now state Lascoux's main result~\cite[Proposition 4.1]{lascoux}.
Recall from the introduction the definition of the (strong) Bruhat order $\order$ in $\Set$.
 The next result appears for instance in~\cite[Lemma]{verma} or~\cite[Theorem 1.1.(II)(ii)]{deodhar}.

\begin{lemma}
\label{lemma:deodhar}
Let $(W,S)$ be a Coxeter system, let $\ell$ be the length function and  let $\order$ be the strong Bruhat order. Let $w,w' \in W$ and let $s \in S$ such that $\ell(sw) = \ell(w)-1$ and $\ell(sw') = \ell(w')-1$. Then:
\[
w \order w' \iff sw \order w' \iff sw \order sw'.
\]
\end{lemma}

\begin{theorem}[\protect{Lascoux~\cite[Proposition 4.1]{lascoux}}]
\label{theorem:lascoux}
Let $\charge \in \{0,\dots,e-1\}$. 
Let $w,w' \in \GAi$ and let $\lambda,\lambda'$ respectively be the  associated $\charge$-charged $e$-cores. Then:
\[
w \order w' \iff \Y(\lambda) \subseteq \Y(\lambda').
\]
\end{theorem}

\begin{proof}
We use induction on $\ell(w,w') \coloneqq \ell(w) + \ell(w')$. If $\ell(w,w') = 0$ then $\ell(w) = \ell(w') = 0$ thus $w = w' = 1$ and $\Y(\lambda) = \Y(\lambda') = \emptyset$ and the equivalence is true.

Now let $n \geq 0$ and assume that the result holds for all $w,w' \in \GAi$ such that $\ell(w,w') \leq n$. Let $w,w' \in \GAi$ such that $\ell(w,w') = n+1$ and let $\lambda,\lambda'$ be the corresponding $\charge$-charged $e$-cores. If
$w = 1$ or $w' = 1$ then the result is immediate since $\Y(\lambda) = \emptyset$ or $\Y(\lambda') = \emptyset$ respectively. Thus, we assume that $w,w' \neq 1$. In particular, there exists $i \in \{0,\dots,e-1\}$ such that with $s = s_i$ we have $\ell(s w') < \ell(w')$. By Lemma~\ref{lemma:siw_GA}, we have $sw' \in \GAi$.
\begin{itemize}
\item Assume first that $\ell(sw) = \ell(w) - 1$. By Lemma~\ref{lemma:siw_GA} again, we have $sw \in \GAi$. Since $\ell(sw) < \ell(w)$ and $\ell(sw') < \ell(w')$, we can use Lemma~\ref{lemma:deodhar} to have:
\[
w \order w' \iff sw \order sw',
\]
moreover $\ell(sw,sw') = \ell(w,w')-2 < \ell(w,w')$ thus we can use the induction  hypothesis to obtain:
\[w \order w' \iff \Y(s\lambda) \subseteq \Y(s\lambda').
\]
Since, again, $\ell(sw) < \ell(w)$ and $\ell(sw') < \ell(w')$, by Proposition~\ref{proposition:removable_node_siw} we know that~$\lambda$ and~$\lambda'$ both have a removable $i$-node. Hence, by Proposition~\ref{proposition:cores_AR} we know that~$\lambda$ or~$\lambda'$ cannot have any addable $i$-node, hence by Lemma~\ref{lemma:action_si_inclusion} we deduce that:
\[
w \order w' \iff \Y(\lambda) \subseteq \Y(\lambda'),
\]
as desired.
\item We now assume that $\ell(sw) = \ell(w) + 1$. Setting $\tilde w \coloneqq sw$ and $\tilde\lambda \coloneqq s\lambda$, we have  $s\tilde w = w$ and $\ell(s\tilde w) < \ell(\tilde w)$.  Recalling $\ell(sw') < \ell(w')$, we can thus use the induction hypothesis for $s\tilde w$ and $sw'$ to find:
\[
s\tilde w \order sw' \iff \Y(s\tilde \lambda) \subseteq \Y(s \lambda').
\]
As in the preceding point, we know that neither $\tilde\lambda$ or $\lambda'$ has an addable $i$-node and thus:
\[
s\tilde w\order sw' \iff \Y(\tilde \lambda) \subseteq \Y(\lambda').
\]
Since $\lambda'$ has no addable $i$-nodes, we deduce from Lemma~\ref{lemma:mixed} that:
\[
s\tilde w \order sw' \iff \Y(s\tilde \lambda) \subseteq \Y(\lambda').
\]
We now note that since $\ell(s\tilde w) < \ell(\tilde w)$ and $\ell(sw') < \ell(w')$, by Lemma~\ref{lemma:deodhar} we have $s\tilde w \order sw' \iff s\tilde w \order w'$  thus we conclude the proof since $s\tilde w = w$.
\end{itemize}
\end{proof}

\begin{remark}
Using crystals theory, Jacon--Lecouvey~\cite[Proposition 5.18]{jacon-lecouvey:keys} obtained the same result as Theorem~\ref{theorem:lascoux}, via reduction to the type A case.
\end{remark}

\section{Strong Bruhat order on the affine symmetric group}
\label{section:bruhat_whole_affine}

Our aim is now to give an efficient description of the strong Bruhat order on the whole affine symmetric group. 

\subsection{Product characterisation}

For $w \in \Set$ and $J \subseteq \{0,\dots,e-1\}$, recall from Proposition~\ref{proposition:representative} that $w^J$ denotes the minimal length coset representative associated with $w$ for the parabolic subgroup generated by $s_i$ for $i \in J$. 

\begin{proposition}[\protect{Deodhar~\cite[Lemma 3.6]{deodhar}}]
\label{proposition:deodhar_Ji}
Let $\{J_i\}_i$ be a family of subsets of $\{0,\dots,e-1\}$ such that $\cap_i J_i = \emptyset$. Then for any $w,w' \in \Set$ we have $w \order w'$ if and only if $w^{J_i} \order w'^{J_i}$ for all $i$.
\end{proposition}

Now recall from Definitions~\ref{definition:GA} and~\ref{definion:w^c} that for $w \in \Set$ and $\charge \in \{0,\dots,e-1\}$, we denote by $w^{(\charge)}$ the element $w^{\setc}$, where $\setc = \{0,\dots,e-1\} \setminus\{\charge\}$.

\begin{corollary}
\label{corollary:lascoux_whole}
Let $w,w' \in \Set$.  For any $\charge \in \{0,\dots,e-1\}$, let $\lambda^{(\charge)},\lambda'^{(\charge)}$ be the $\charge$-charged $e$-cores associated with $w,w'$ respectively. Then:
\[
w\order w' \iff \Y\bigl(\lambda^{(\charge)}\bigr) \subseteq \Y\bigl(\lambda'^{(\charge)}\bigr) \text{ for all } \charge \in \{0,\dots,e-1\}.
\]
\end{corollary}

\begin{proof}
For any $\charge \in \{0,\dots,e-1\}$ we have $\charge \notin \setc$ thus $\cap_{\charge=0}^{e-1}\setc = \emptyset$. Applying Proposition~\ref{proposition:deodhar_Ji}, we thus have that $w \order w' \iff w^{(\charge)} \order w'^{(\charge)}$ for all $\charge \in \{0,\dots,e-1\}$. By definition we have $\lambda^{(\charge)} = w\cdot \emptypart$ and by Proposition~\ref{proposition:action_w_wcharge} we deduce  that $\lambda^{(\charge)} =  w^{(\charge)}\cdot\emptypart$. The result now follows from applying Theorem~\ref{theorem:lascoux} to the elements $w^{(\charge)} \in \GAi$ for all $\charge \in \{0,\dots,e-1\}$.
\end{proof}

A natural way to see the result of Corollary~\ref{corollary:lascoux_whole} is to make $\Set$ acting  diagonally on the  $e$-tuple $(\emptypart[0],\dots,\emptypart[e-1])$. The $e$-tuples of $e$-cores that we obtain this way are a the \emph{$(e,\mcharge)$-cores} of Jacon--Lecouvey~\cite{jacon-lecouvey:keys}, where in our case $\mcharge=  (0,1,\dots,e-1)$. We will give some interpretations of our results in terms of $(e,\mcharge)$-cores in Appendix~\ref{appendix_section:cores_higher_levels}.

\subsection{Inductive construction on abaci}
\label{subsection:inductive_abaci}

When applying Corollary~\ref{corollary:lascoux_whole}, it seems that for each $\charge \in \{0,\dots,e-1\}$ we have to compute the $e$-core $\lambda^{(\charge)}$ through the action of $w \in \Set$ on the empty $\charge$-charged $e$-core $\emptypart$. We will give here an easy inductive construction, to compute $\lambda^{(\charge+1)}$ from $\lambda^{(\charge)}$.

\begin{definition}
\label{definition:procedure}
Let $w \in \Set$. Define the $e$-tuple of abaci $A^w = (A_0,\dots,A_{e-1})$ as follows.
\begin{itemize}
\item We have $A_0 = w(\mathbb{Z}_{\leq 0})$, that is, the abacus $A_0$ is the abacus of the $0$-charged $e$-core~$w\cdot\emptypart[0]$.
\item For $\charge \in \{1,\dots,e-1\}$, the abacus $A_{\charge}$ is obtained from the abacus $A_{\charge-1}$ by adding a bead at $w(\charge)$.
\end{itemize}
\end{definition}

\begin{remark}
\label{remark:add_bead_w(i)}
In the setting of Definition~\ref{definition:procedure}, for $\charge \in \{1,\dots,e-1\}$ the abacus $A_{\charge-1}$ has a gap at $w(\charge)$ indeed, since by Proposition~\ref{proposition:abacus_w} we know that $w\cdot\emptypart[0]$ has gaps at $w(1),\dots,w(e-1)$. In fact, by Proposition~\ref{proposition:abacus_w} we know that for any $i \in \{\charge,\dots,e-1\}$, in $A_{\charge-1}$ the smallest gap with residue $w(i) \pmod{e}$ is exactly at $w(i)$.
\end{remark}

\begin{example}
\label{example:procedure}
Take $e = 4$ and $w = s_1 s_2 s_0 \in \Set[4]$  so that we have $w = [0,3,1,6]$ in the window notation. By Proposition~\ref{proposition:abacus_w}, the associated partition has abacus:
\[
A_0 = \bigl(1+4\mathbb{Z}_{<0}\bigr) \sqcup \bigl(6+4\mathbb{Z}_{< 0}\bigr) \sqcup \bigl(3+4\mathbb{Z}_{< 0}\bigr) \sqcup \bigl(0+4\mathbb{Z}_{< 0}\bigr).
\]
Representing the above gaps at $1,6,3,0$ with a red diamond, the abacus $A_0$ is thus:
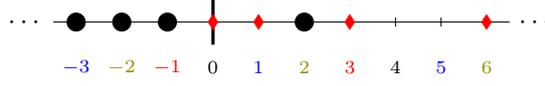
\begin{center}
\begin{tikzpicture}[scale=\scaletikz]
\emptyabacusinside{-3}{6}
\bbead{2}
\bbead{-1}
\bbead{-2}
\bbead{-3}
\loz{0}
\loz{1}
\loz{3}
\loz{6}

\foreach \i in {0,4} {\num{\i}{\cquatre}}
\foreach\i in {-1,3}{\num{\i}{\ctrois}}
\foreach \i in {-2,2,6}{\num{\i}{\cdeux}}
\foreach \i in {-3,1,5}{\num{\i}{\cun}}
\end{tikzpicture}
\end{center}
We now successively add a bead at $w(\charge)$ to $A_0$ for $\charge \in \{1,\dots,3\}$. For each step, the bead that we add is in red.
We add $w(1) = 0$ to obtain:
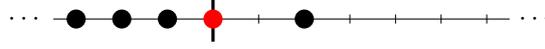
\begin{center}
\begin{tikzpicture}[scale=\scaletikz]
\emptyabacusinside{-3}{6}
\bbead{2}
\bead{0}
\bbead{-1}
\bbead{-2}
\bbead{-3}
\end{tikzpicture}
\end{center}
which is the abacus $A_1$.  We now add $w(2) = 3$ to $A_1$ to obtain:
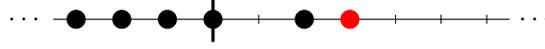
\begin{center}
\begin{tikzpicture}[scale=\scaletikz]
\emptyabacusinside{-3}{6}
\bbead{2}
\bbead{0}
\bead{3}
\bbead{-1}
\bbead{-2}
\bbead{-3}
\end{tikzpicture}
\end{center}
which is the abacus $A_2$. Finally, we add $w(3) = 1$ to $A_2$ to obtain:
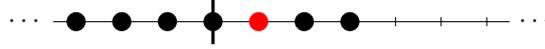
\begin{center}
\begin{tikzpicture}[scale=\scaletikz]
\emptyabacusinside{-3}{6}
\bbead{2}
\bbead{0}
\bbead{3}
\bead{1}
\bbead{-1}
\bbead{-2}
\bbead{-3}
\end{tikzpicture}
\end{center}
which is the abacus $A_3$. Note that the abaci $A_0,A_1,A_2,A_3$ correspond to the partitions $(2),(1),(1,1),\emptyset$ respectively.
\end{example}

For $\charge \in \{0,\dots,e-1\}$ an affine Grassmannian element $w \in \GAi$, the $e$-tuple of abaci $A^w = (A_0,A_1,\dots,A_{e-1})$ of Definition~\ref{definition:procedure}  can be easily constructed from $A_\charge$. We explicit this construction below in the case $\charge = 0$ (the general case being similar).

\begin{proposition}
\label{proposition:procedure_grassmannian}
Let $A'_0$ be an abacus of charge $0$.
\begin{itemize}
\item Let $A'_1$ be the abacus obtained from $A'_0$ by adding a bead in the leftmost gap; let $j_1$ be the $e$-residue of this bead.
\item For $\charge$ from $2$ to $e-1$, let $A'_\charge$ be the abacus obtained from $A'_{\charge-1}$ by adding a bead in the leftmost gap with $e$-residue not in $\{j_1,\dots,j_{\charge-1}\}$. Let $j_\charge$ be the $e$-residue of this bead.
\end{itemize}
Let $w \in \GAi[0]$ be the affine Grassmannian element associated with $A'_0$. Then $(A'_0,\dots,A'_{e-1})$ coincide with the $e$-tuple of abaci $A^w$ of Definition~\ref{definition:procedure}.
\end{proposition}

Note that the element $w \in \GAi[0]$ of Proposition~\ref{proposition:procedure_grassmannian} is well-defined (and unique) by Corollary~\ref{corollary:GA_cores_bij}.

\begin{proof}
By construction we have $A_0 = A'_0$. Recall from  Proposition~\ref{proposition:abacus_w} that $A_0 = \sqcup_{i = 1}^e \bigl(w(i) + e\mathbb{Z}_{< 0}\bigr)$. By Proposition~\ref{proposition:charact_GA} we have $w(1) < \dots < w(e)$ thus the leftmost gap in $A_0$ is $w(1)$. Thus, we have $j_1 = w(1) \pmod{e}$ and:
\[
A'_1 = A_1 = \bigl(w(1) + e\mathbb{Z}_{\leq 0}\bigr) \sqcup \bigsqcup_{i = 2}^{e-1}\bigl(w(i) + e\mathbb{Z}_{< 0}\bigr).
\] 
Now assume that for some $\charge \in \{2,\dots,e-1\}$ we have $j_i = w(i) \pmod{e}$ for any $i \in \{1,\dots,\charge-1\}$ and:
\[
A'_{\charge-1} = A_{\charge-1} = \bigsqcup_{i = 1}^{\charge-1} \bigl(w(i) + e\mathbb{Z}_{\leq 0}\bigr) \sqcup \bigsqcup_{i = \charge}^{e-1} \bigl(w(i) + e\mathbb{Z}_{< 0}\bigr).
\]
Since $w(\charge) < \dots < w(e)$, by Lemma~\ref{lemma:w(i)_congr} we know that the leftmost gap in $A_{\charge-1}$ with residue different from $j_1,\dots,j_{\charge-1}$ is $w(\charge)$. Thus, we have $j_\charge = w(\charge) \pmod{e}$ and:
\[
A'_{\charge} = A_{\charge} = \bigsqcup_{i = 1}^{\charge} \bigl(w(i) + e\mathbb{Z}_{\leq 0}\bigr) \sqcup \bigsqcup_{i = \charge+1}^{e-1} \bigl(w(i) + e\mathbb{Z}_{< 0}\bigr).
\]
By induction, this proves the desired result.
\end{proof}

\begin{example}
\label{example:procedure_abacus_GA}
Take $e = 4$ and let us consider $A'_0$ to be the following abacus:
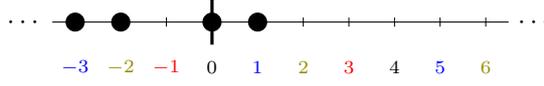
\begin{center}
\begin{tikzpicture}[scale=\scaletikz]
\emptyabacusinside{-3}{6}
\bbead{1}
\bbead{0}
\bbead{-2}
\bbead{-3}

\foreach \i in {0,4} {\num{\i}{\cquatre}}
\foreach\i in {-1,3}{\num{\i}{\ctrois}}
\foreach \i in {-2,2,6}{\num{\i}{\cdeux}}
\foreach \i in {-3,1,5}{\num{\i}{\cun}}
\end{tikzpicture}
\end{center}
Note that $A'_0$ has charge $0$ indeed.  The first gap of $A'_0$ is at $-1$ thus $j_1 = -1 = 3 \in \mathbb{Z}/4\mathbb{Z}$ and $A'_1$ is the following abacus:
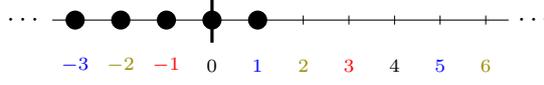
\begin{center}
\begin{tikzpicture}[scale=\scaletikz]
\emptyabacusinside{-3}{6}
\bbead{1}
\bbead{0}
\bbead{-2}
\bbead{-3}
\bbead{-1}

\foreach \i in {0,4} {\num{\i}{\cquatre}}
\foreach\i in {-1,3}{\num{\i}{\ctrois}}
\foreach \i in {-2,2,6}{\num{\i}{\cdeux}}
\foreach \i in {-3,1,5}{\num{\i}{\cun}}
\end{tikzpicture}
\end{center}
Now the first gap in $A'_1$ is at $2$ and $2 \pmod{4} \notin \{j_1\}$ thus $j_2 = 2 \in \mathbb{Z}/4\mathbb{Z}$ and $A'_2$ is the following abacus:
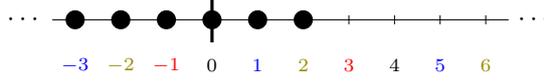
\begin{center}
\begin{tikzpicture}[scale=\scaletikz]
\emptyabacusinside{-3}{6}
\bbead{2}
\bbead{1}
\bbead{0}
\bbead{-2}
\bbead{-3}
\bbead{-1}

\foreach \i in {0,4} {\num{\i}{\cquatre}}
\foreach\i in {-1,3}{\num{\i}{\ctrois}}
\foreach \i in {-2,2,6}{\num{\i}{\cdeux}}
\foreach \i in {-3,1,5}{\num{\i}{\cun}}
\end{tikzpicture}
\end{center}
Finally, the first gap in $A'_2$ is at $3$ but $3 \pmod{4} = j_1 \in \{j_1,j_2\}$. We thus consider the next gap, which is at $4$. We have $4 \pmod{4} \notin \{j_1,j_2\}$ thus $j_3 = 4 = 0 \in \mathbb{Z}/4\mathbb{Z}$ and $A'_3$ is the following abacus:
\begin{center}
\begin{tikzpicture}[scale=\scaletikz]
\emptyabacusinside{-3}{6}
\bbead{4}
\bbead{2}
\bbead{1}
\bbead{0}
\bbead{-2}
\bbead{-3}
\bbead{-1}

\foreach \i in {0,4} {\num{\i}{\cquatre}}
\foreach\i in {-1,3}{\num{\i}{\ctrois}}
\foreach \i in {-2,2,6}{\num{\i}{\cdeux}}
\foreach \i in {-3,1,5}{\num{\i}{\cun}}
\end{tikzpicture}
\end{center}

Now by Propositions~\ref{proposition:abacus_w} and~\ref{proposition:charact_GA} we know that $A'_0$ is the abacus associated with the affine Grassmannian element $w \in \GAi[0]$ with window notation $[-1,2,4,5]$. Then $A'_1,A'_2,A'_3$ were constructed from $A'_0 = A_0$ by adding a bead at $-1,2,4$ respectively so that $A'_\charge = A_\charge$ for any $\charge\in \{0,\dots,e-1\}$ as expected. Note that $A_0$ is the abacus of the $4$-charged partition $(1,1)$, which has Young diagram $\ytableaushort{0,3}$ (where in each box we have indicated the corresponding residue) so that $w = s_3 s_0$.
\end{example}

We now give the main result of this part.

\begin{proposition}
\label{proposition:procedure}
Let $w \in \Set$ and let $A^w = (A_0,\dots,A_{e-1})$ be the $e$-tuple of abaci of Definition~\ref{definition:procedure}. For any $\charge \in \{0,\dots,e-1\}$, the abacus of $w\cdot\emptypart$ is $A_{\charge}$.
\end{proposition}

\begin{proof}
Let $\charge = \{0,\dots,e-1\}$. The statement is true for $\charge=0$, indeed,  since the abacus of $\emptypart[0]$ is $\mathbb{Z}_{\leq 0}$ we have $A_0 = w(\mathbb{Z}_{\leq 0})$.

We now assume $\charge \neq 0$. Let $A'_{\charge}$ be the abacus of $w\cdot\emptypart$. The abacus of $\emptypart$ is $\mathbb{Z}_{\leq \charge} = \mathbb{Z}_{\leq 0} \sqcup \{1,\dots,\charge\}$ thus:
\[
A'_{\charge} = w(\mathbb{Z}_{\leq 0}) \sqcup \{w(1),\dots,w(\charge)\}=A_0 \sqcup \{w(1),\dots,w(\charge)\}.
\]
Hence, the abacus $A'_{\charge}$ is obtained from $A_0$ by adding a bead successively at $w(1),\dots,w(\charge)$,
thus $A'_{\charge} = A_{\charge}$ by definition.
\end{proof}

\begin{example}
\label{example:action_si_tuple}
Take $e = 4$ and consider the element $w = s_1 s_2 s_0 \in \Set[4]$ of Example~\ref{example:procedure}. Writing the residues in the boxes, we have successively:
\[
\begin{array}{*7l}
\emptypart[0] &\xrightarrow{s_0}& \ytableaushort{0} &\xrightarrow{s_2} & \ytableaushort{0} &\xrightarrow{s_1}& \ytableaushort{01}\, ,
\\
\emptypart[1] &\xrightarrow{s_0}& \emptypart[1]&\xrightarrow{s_2}&\emptypart[1]&\xrightarrow{s_1}&\ytableaushort{1}\, ,
\\
\emptypart[2]&\xrightarrow{s_0}&\emptypart[2]&\xrightarrow{s_2}&\ytableaushort{2}&\xrightarrow{s_1}&\ytableaushort{2,1}\, ,
\\
\emptypart[3] &\xrightarrow{s_0}&\emptypart[3]&\xrightarrow{s_2}&\emptypart[3]&\xrightarrow{s_1}&\emptypart[3].
\end{array}
\]
We recover the partitions that we obtained in Example~\ref{example:procedure} indeed.
\end{example}

Recall from Corollary~\ref{corollary:lascoux_whole} that given $w \in \Set$ we are interested in computing the $e$-cores $w\cdot\emptypart$ for $\charge\in\{0,\dots,e-1\}$.  Proposition~\ref{proposition:procedure} drastically simplifies the calculations of the $e$-cores $w\cdot\emptypart[2],\dots,w\cdot\emptypart[e-1]$, since for each $\charge \in \{2,\dots,e-1\}$ it replaces a series of $\ell(w)$ operations (namely, looking at addable $i$-nodes) by just one operation (namely, adding a unique bead).

\subsection{Inductive construction on partitions}
\label{subsection:inductive_construction_partitions}

We now want to perform the procedure of Definition~\ref{definition:procedure} at the level of partitions uniquely, without the use of abaci.

\begin{lemma}
\label{lemma:procedure_two_moves}
Let $A$ be an abacus. If $x \notin A$ and if $A'$ is the abacus obtained from $A$ by adding a bead at $x$, then $A'$ is also obtained from $A$ by:
\begin{enumerate}[$1)$]
\item moving the bead of $A$ at $y \coloneqq \min(\mathbb{Z}\setminus A)-1$ to $x$;
\item adding a bead at $y$.
\end{enumerate}
\end{lemma}

\begin{proof}
It suffices to notice that $A$ has a bead at $\min(\mathbb{Z}\setminus A)-1$ indeed, but this is clear since with $z \coloneqq \min (\mathbb{Z}\setminus A)$ then by definition $z$ is the leftmost gap in $A$, thus $z-1 \in A$ is a bead.
\end{proof}

In the context of Lemma~\ref{lemma:procedure_two_moves}, adding a bead at $y$ in the abacus $A'' \coloneqq A' \setminus\{y\}$ means removing the first column on the partition corresponding to $A''$. The meaning of the first operation in Lemma~\ref{lemma:procedure_two_moves} is given by the notion of  \emph{rim hook} (see, for instance, \cite[\textsection 2.7]{james-kerber}). 

\begin{definition}
Let $\lambda$ be a partition.
\begin{itemize}
\item A \emph{rim hook} of $\lambda$ is a subset of $\Y(\lambda)$ of the following form:
\[
R_\gamma^\lambda \coloneqq \bigl\{(x,y) \in \Y(\lambda) : x \geq a \text{ and } y \geq b \text{ with } (x+1,y+1) \notin \Y(\lambda)\bigr\},
\]
where $\gamma = (a,b) \in \Y(\lambda)$. If $h \coloneqq \#R_\gamma^\lambda$ we say that $R_\gamma^\lambda$ is an $h$-rim hook of $\lambda$.
\item Let $\gamma = (a,b) \in \Y(\lambda)$. The \emph{foot} (resp. \emph{hand}) of the rim hook $R_\gamma^\lambda$ is the node $(x,b) \in R_\gamma^\lambda$ (resp. $(a,y) \in R_\gamma^\lambda$ with $x \geq a$ (resp. $y \geq b$) maximal.
\end{itemize}
\end{definition}

\begin{example}
\label{example:action_si_different_residues}
We consider the partition $\lambda \coloneqq (4,4,3,3,2,1)$. An example of an $8$-rim hook is:
\[
R^\lambda_{(2,1)} =
\begin{ytableau}
{}&{}&{}&{}
\\
\gamma&{}&\times&h
\\
{}&{}&\times
\\
{}&\times&\times
\\
\times&\times
\\
f
\end{ytableau}\, 
\]
(the rim hook is the set of boxes with either $\times$, $f$ or $h$ inside).
The node $(2,1)$ has a $\gamma$ inside,  the foot is the node $(6,1)$ and has an $f$ inside, the hand is the node $(2,5)$ and has an $h$ inside.
\end{example}

Let $\lambda$ be a partition and let $\gamma = (a,b) \in \Y(\lambda)$. Then  $\Y(\lambda)\setminus R$ is the Young diagram of a partition $\mu$. We say that $\mu$ is obtain by \emph{removing} the rim hook $R$ from $\lambda$. We also say that~$\lambda$ is obtained from $\mu$ by \emph{adding} the rim hook $R$ from $\mu$ (note that $R$ is note a rim hook of $\mu$). 

\begin{lemma}
\label{lemma:number_nodes_rimhook_first_column}
Let $\lambda$ be a partition, let $\gamma \in \Y(\lambda)$ and let $\mu$ be the partition obtained from $\lambda$ by removing the rim hook $R_\gamma^\lambda$. Then the number of parts of $\mu$ is the number of parts of $\lambda$ minus the number of nodes of $R_\gamma^\lambda$ in the first column of $\lambda$. 
\end{lemma}

\begin{proof}
It suffices to note that the number of parts of a partition is the length of the first column of its Young diagram.
\end{proof}

The next result is a generalisation of Proposition~\ref{proposition:add_inode_abacus}  (see, for instance, \cite[2.7.13 Lemma]{james-kerber}). Note that we use the term ``hand residue'' of a rim hook to denote the residue of the hand.

\begin{proposition}
\label{proposition:remove_hook_abacus}
Let $\lambda$ be charged partition and let $A$ be its abacus. Let $h \in \mathbb{Z}_{\geq 1}$ and let $i \in \mathbb{Z}/e\mathbb{Z}$. The charged partition $\lambda$ has a removable $h$-rim hook with hand residue $i$ if and only if there exists $x \in A$ with $x-h \notin A$ and $x \equiv i+1 \pmod{e}$. More precisely:
\begin{itemize}
\item If $\lambda$ has an $h$-rim hook $R$ with hand residue $i$ then the abacus of the charged partition that we obtain by removing $R$ from $\lambda$ has the form $(A \setminus\{x\})\sqcup \{x-h\}$ for some $x \in A$ with $x-h \notin A$ and $x \equiv i+1\pmod{e}$.
\item If $x \in A$ is such that $x-h \notin A$ and $x \equiv i+1 \pmod{e}$ then $\lambda$ has an $h$-rim hook $R$ with hand residue $i$ such that $(A \setminus\{x\})\sqcup \{x-h\}$ is the abacus of the charged partition obtained by removing $R$ from $\lambda$.
\end{itemize}
\end{proposition}

\begin{remark}
The result of Proposition~\ref{proposition:remove_hook_abacus} can be easily understood by interpreting the abacus as a series of up and right step along the rim of the Young diagram. The assertions about the residues follow from Lemma~\ref{lemma:inodes_abacus}.
\end{remark}

Recall from Lemma~\ref{lemma:procedure_two_moves} the particular operations that we are interested in.

\begin{lemma}
\label{lemma:move_leftmost}
Let $\lambda$ be a non-empty charged partition and let $A$ be its abacus. Let $m \coloneqq \min(\mathbb{Z}\setminus A)$ be the leftmost gap in $A$ and let $x \in A \cap \mathbb{Z}_{> m}$. Let $h \coloneqq x-m \geq 1$ and let $R$ be the rim hook corresponding to $x \in A$ and $m=x-h \notin A$ as in Proposition~\ref{proposition:remove_hook_abacus}. Then $R$ has exactly one node in the first column of $\Y(\lambda)$ if and only if $x= m+1$ or $m+1 \notin A$.
\end{lemma}

\begin{proof}
Recall that there is a one-to-one correspondence between the beads of an abacus with a gap somewhere to their left and the non-zero parts of the associated charged partition. Let $B \coloneqq (A \setminus\{x\})\sqcup \{m\}$ and let $\mu$ be the associated charged partition.

Assume first that $m+1 \in A$. If $x = m+1$ then $h = x-m = 1$ and $R$ is reduced to only one node. By minimality of $m$ this node is the only node of its row, whence the result. We now assume $x \neq m+1$.
Then in $A$, the beads at $x$ and $m+1$ both correspond to a non-zero part of $\lambda$  since $m \notin A$ (noting that $x > m$ by definition), but then in $B$ the beads at $m$ and $m+1$ both correspond to zero parts of $\mu$ since then $\mathbb{Z}_{\leq m+1} \subseteq B$. Hence if $x \neq m+1$ then $\mu$ has at least two parts less than $\lambda$ thus $R$ has at least two nodes in the first column of $\Y(\lambda)$ by Lemma~\ref{lemma:number_nodes_rimhook_first_column}. 

Conversely, assume that $m+1 \notin A$. Then for any $y \in A \cap \mathbb{Z}_{> m}$ we know that $y \in A$ has a gap somewhere to its left (namely, at $m+1$). By construction, if $y \neq x$ we also have $y \in B$ and $m+1 \notin B$ thus $y \in B$ has a gap somewhere to its left. Now to the bead $x \in A$ corresponds a non-zero part in $\lambda$, but by minimality $m \in B$ has no gaps to its left thus $m$ corresponds to a zero part of $\mu$. Hence $\mu$ has exactly one part less than $\lambda$ and we conclude the proof by Lemma~\ref{lemma:number_nodes_rimhook_first_column}.
\end{proof}

\begin{remark}
In Lemma~\ref{lemma:move_leftmost}, the bead $x \in A$ has a gap somewhere to its left and we move the bead $x$ as far left as possible.
\end{remark}

We can now give our main theorem.

\begin{theorem}
\label{theorem:procedure_partitions}
Let $w \in \Set$. For any $\charge \in \{1,\dots,e-1\}$, the $e$-core $w\cdot\emptypart$ can be obtained from $\lambda^{(\charge-1)} = w\cdot\emptypart[\charge-1]$ by:
\begin{enumerate}[$1)$]
\item constructing the $(\charge-1)$-charged partition $\mu_{\charge-1,\charge}$ obtained from $\lambda^{(\charge-1)}$ by adding the smallest rim hook with the following properties:
\begin{itemize}
\item it has only one node below the first column of $\Y(\lambda^{(\charge-1)})$,
\item its hand has residue $w(\charge)-1$,
\end{itemize}
\item removing the first column of $\mu_{\charge-1,\charge}$.
\end{enumerate}
\end{theorem}

\begin{proof}
Let $\charge \in \{1,\dots,e-1\}$. By Proposition~\ref{proposition:procedure} and Lemma~\ref{lemma:procedure_two_moves}, we know that the abacus~$A_{\charge}$ of $\lambda^{(\charge)} = w\cdot \emptypart$ is obtained from the abacus $A_{\charge-1}$ of $\lambda^{(\charge-1)}$  by:
\begin{enumerate}[$1)$]
\item constructing the abacus $B$ obtained by moving the bead of $A_{\charge-1}$ at $y_\charge \coloneqq\min(\mathbb{Z}\setminus A_{\charge-1}) - 1$ to $w(\charge)$,
\item adding a bead  at $y_\charge$ in $B$.
\end{enumerate}
Let $\mu$ be the charged partition associated with $B$. Note that $y_\charge+1$ is the leftmost gap of $A_{\charge-1}$, thus since $w(\charge)$ is also a gap in $A_{\charge-1}$ (by Remark~\ref{remark:add_bead_w(i)}) we have $y_\charge < w(\charge)$. Hence, the abacus $A_{\charge-1}$ is obtained from $B$ by moving the bead at $w(\charge)$ to the left to $y_\charge$, which is by definition the leftmost gap of $B$. 

Now if $y_\charge+1 \in B$ then by construction, since $y_\charge+1 \notin A_{\charge-1}$ it means that $y_\charge+1 = w(\charge)$. Thus, we can apply Lemma~\ref{lemma:move_leftmost} to find that $A_{\charge-1}$ is obtained from $B$ by removing a rim hook~$R$ with exactly one node below the first column of $\Y(\mu)$. 
Moreover, this rim hook corresponds to the bead at $w(\charge)$ in $B$ (and gap $y_\charge$) thus by Proposition~\ref{proposition:remove_hook_abacus} we know that the hand residue of $R$ is $w(\charge)-1 \pmod{e}$.

It remains to prove that this rim hook $R$ is the shortest possible. Let $S$ be another rim hook of that one can add to $\lambda^{(\charge-1)}$ such that $S$ has hand residue $w(\charge)-1\pmod{e}$ and only one node below the first column of  $\lambda^{(\charge-1)}$. Then adding $S$ to $\lambda^{(\charge-1)}$  corresponds to moving the bead of $A_{\charge-1}$ at $y_\charge$ to a gap $z \notin A_{\charge-1}$ with $z \equiv w(\charge) \pmod{e}$. By Remark~\ref{remark:add_bead_w(i)}, we know that necessarily $z > w(\charge)$, thus by Proposition~\ref{proposition:remove_hook_abacus} we know that $S$ is longer than $R$.

For the second point of the procedure, since $y_\charge$ is the smallest gap in $B$, we know that adding a bead there means removing the first column of the corresponding partition and this concludes the proof. Note that removing the first column changes the charge from $\charge-1$ to $\charge$ (since the charge is given by the residue of the node $(1,1)$).
\end{proof}

\begin{remark}
As in Proposition~\ref{proposition:procedure_grassmannian}, if $w \in \GAi[0]$ then during the procedure of Theorem~\ref{theorem:procedure_partitions} one just want to add rim hooks ending at a node whose residue has not already been obtained this way. In particular, the $1$-charged $e$-core $\lambda^{(1)}$ is always obtained by removing the first column of $\lambda^{(0)}$.
\end{remark}

\begin{example}
Take $e = 4$ and consider the element $w = s_1s_2s_0 \in \Set$, which has window notation $[0,3,1,6]$ (cf. Example~\ref{example:procedure}). We saw that $w \cdot \emptypart[0]$ is the $e$-core $\lambda^{(0)} = (2)$. We then add the smallest rim hook with only one node in the first column and hand residue $w(1)-1 = -1 \equiv 3 \pmod{4}$ to find:
\[
\begin{ytableau}
 0&
\\
*(red)3
\end{ytableau}
\]
(the charge is recalled in the first box, the rim hook that we add is in red and the hand residue is written in the corresponding box).
Removing the first column gives the partition $\lambda^{(1)} = (1)$. We do the procedure again, noting that now the charge became $1$, wanting a rim hook with only one node in the first column and hand residue $w(2) -1 = 2\pmod{4}$:
\[
\begin{ytableau}
 1 &*(red)2
\\
*(red)&*(red)
\end{ytableau},
\]
thus removing the first column gives the partition $\lambda^{(2)} = (1,1)$. Now we want a rim hook with only one node in the first column and hand residue $w(2) - 1 = 0\pmod{4}$:
\[
\begin{ytableau}
 2
\\
{}
\\
*(red)0
\end{ytableau}
\]
thus removing the first column gives the empty partition $\lambda^{(3)}$. If we want we can go on with the procedure, looking for a rim hook with only one node in the first column and hand residue $w(3) -1 = 5 \equiv 1$ (the charge is now $3$):
\[
\begin{ytableau}
*(red)& *(red) & *(red)1
\end{ytableau}
\]
thus removing the first column gives the partition $\lambda^{(0)} = (2)$ back.
\end{example}

\newpage

\appendix

\section{Generalised cores of Jacon--Lecouvey}
\label{appendix_section:cores_higher_levels}

Let $\level \geq 2$ and let $\mcharge = (\charge_1,\dots,\charge_\ell) \in \mathbb{Z}^\level$.

\begin{definition}[Jacon--Lecouvey \cite{jacon-lecouvey:keys}]
For each $k \in \{1,\dots,\ell\}$, let $\lambda^{(k)}$ be a $\charge_k$-charged partition and let $A_k$ be its abacus.
 We say that $\mpart = (\lambda^{(1)},\dots,\lambda^{(\ell)})$ is an \emph{$(e,\mcharge)$-core} if:
 \[
 A_1 \subseteq A_2 \subseteq \dots \subseteq A_\ell \subseteq A_1 + e.
 \]
\end{definition}

Note that if $\mpart$ is an $(e,\mcharge)$-core then each component of $\mpart$ is an $e$-core. A basic example of $(e,\mcharge)$-core is the following.

\begin{lemma}[\protect{\cite[Lemma 5.12]{jacon-lecouvey:keys}}]
\label{lemma:mempty_escore}
Assume that $\charge_1 \leq \charge_2 \leq \dots \leq \charge_\ell \leq \charge_1 + e$. Then $(\emptypart[\charge_1],\dots,\emptypart[\charge_\ell])$ is an $(e,\ell)$-core.
\end{lemma}

\begin{remark}
Note that the result of~\cite{jacon-lecouvey:keys} is stated for any $\mcharge \in \mathbb{Z}^\ell$, however the proof is only valid under the above condition (which is necessarily satisfied for $(e,\mcharge)$-cores, see~\cite[Proposition 2.14]{jacon-lecouvey:cores}).
\end{remark}

 The $(e,\mcharge)$-core condition on $\level$-tuple of abaci is compatible with the diagonal action of $\Set$, thus we deduce the following result.

\begin{proposition}[\protect{\cite[Proposition 5.14]{jacon-lecouvey:keys}}]
The affine symmetric group $\Set$ acts on the set of $(e,\mcharge)$-cores.
\end{proposition}

If $\mpart = (\lambda^{(1)},\dots,\lambda^{(\ell)})$, we say that $\mpart$ has an addable (resp. a removable) node if there exists $k \in \{1,\dots,\ell\}$ such that $\lambda^{(k)}$ has an addable (resp. a removable) node.
The next result is similar to Proposition~\ref{proposition:cores_AR} (and so are their proofs).

\begin{lemma}[\protect{\cite[Lemma 5.16]{jacon-lecouvey:keys}}]
\label{lemma:addable_noremovable_escores}
Let $(\mparti{1},\dots,\mparti{\level})$ be an $(e,\mcharge)$-core and let $i \in \{0,\dots,e-1\}$.
\begin{itemize}
\item If $\mpart$ has an addable $i$-node then for all $k\in\{1,\dots,\level\}$, the $e$-core $\mparti{k}$ has no removable $i$-nodes.
\item In particular, if $\mpart$ has a removable $i$-node  then for all $k\in\{1,\dots,\level\}$, the $e$-core $\mparti{k}$ has no addable $i$-nodes.
\end{itemize}
\end{lemma}

%
%
%

 By Lemma~\ref{lemma:addable_noremovable_escores},  the action of the generator $s_i \in \Set$ is either adding all possible $i$-nodes in all components, or removing all possible $i$-nodes in all components.

\begin{example}
\label{example:action_escores}
For $e = 4$ and $\mcharge = (0,1,2,3)$, the action of $s_1s_2s_0 \in \Set$ is obtained on the empty $(e,\mcharge)$-core from Example~\ref{example:action_si_tuple}  by reading column by column.
\end{example}

From now on we set $\mcharge = (0,1,\dots,e-1)$, in particular $\ell = e$. Let $\mempty \coloneqq (\emptypart[0],\dots,\emptypart[e-1])$, which is an $(e,\mcharge)$-core by Lemma~\ref{lemma:mempty_escore}.
The next result is the analogue of Propositions~\ref{proposition:addable_node_siw} and~\ref{proposition:removable_node_siw}.

\begin{proposition}
\label{proposition:length_addable}
Let $w \in \Set$  and let $\mpart\coloneqq w \cdot \mempty$ be the $(e,\mcharge)$-core associated with $w$. Let $i \in \{0,\dots,e-1\}$. We have:
\begin{align*}
\ell(s_i w) < \ell(w) &\iff \mpart \text{ has a removable } i\text{-node},
\\
\intertext{and similarly:}
\ell(s_i w) > \ell(w) &\iff \mpart \text{ has an addable } i\text{-node}.
\end{align*}
In particular, the $(e,\mcharge)$-core $\mpart$ has either an addable or a removable $i$-node (recalling that these two cases cannot happen at the same time).
\end{proposition}

\begin{proof}
Let us write $\mpart = (\lambda^{(0)},\dots,\lambda^{(e-1)})$.
We first prove the two direct implications. 
For any $\charge \in \{0,\dots,e-1\}$, write $A_\charge \coloneqq \mathbb{Z}_{\leq \charge}$ the abacus of the $\charge$-charged empty $e$-core~$\emptypart$, in particular, the abacus of $\lambda^{(\charge)} = w\cdot\emptypart$ is $w(A_\charge)$.
Recall from  Lemma~\ref{lemma:length(ws_i)} that:
\begin{equation}
\label{equation:proof_length(ws_i)}
\ell(s_i w) < \ell(w) \iff \ell(w^{-1}s_i) < \ell(w^{-1}) \iff w^{-1}(i) > w^{-1}(i+1).
\end{equation}

\begin{itemize}
\item
We first assume that $\ell(s_i w ) < \ell(w)$. Let $\charge \in \{0,\dots,e-1\}$ such that $\charge \equiv w^{-1}(i+1)\pmod{e}$ and let $k \in \mathbb{Z}$ such that $\charge = w^{-1}(i+1) + ke$. Since $w \in \Set$, we have $w^{-1} \in \Set$ thus $\charge = w^{-1}(i+1+ke)$. By~\eqref{equation:proof_length(ws_i)} we also obtain:
\begin{align*}
w^{-1}(i)+ke > w^{-1}(i+1)+ke &\text{ thus } w^{-1}(i +ke) > w^{-1}(i+1+ke)
\\
&\text{ thus } w^{-1} (i+ke) > \charge
\\
&\text{ thus } w^{-1}(i+ke) \notin A_{\charge}
\\
&\text{ thus } i+ke \notin w(A_{\charge}).
\end{align*}
Now $w^{-1}(i+1+ke) = \charge \in A_{\charge}$ thus  $i+1+ke \in w(A_{\charge})$. Hence, we have:
\[
i + ke \notin w(A_{\charge}) \text{ and } i+1+ke \in w(A_{\charge}).
\]
We deduce that $\lambda^{(\charge)}$ has a removable $i$-node by Proposition~\ref{proposition:add_inode_abacus}., thus $\mpart$ has a removable $i$-node.
\item Now assume that $\ell(s_i w) > \ell(w)$. Let $\charge \in \{0,\dots,e-1\}$ such that $\charge \equiv w^{-1}(i) \pmod{e}$ and let $k \in \mathbb{Z}$ such that $\charge = w^{-1}(i)+ke$. By similar calculations as above, by~\eqref{equation:proof_length(ws_i)} we have $
w^{-1}(i) < w^{-1} (i+1)$  thus  $\charge < w^{-1}(i+1+ke)$ thus
 $i+1+ke \notin w(A_{\charge})$. Hence, since $i+ke = w(\charge) \in w(A_{\charge})$ we deduce that $\mpart$ has an addable $i$-node.
\end{itemize}
By Lemma~\ref{lemma:addable_noremovable_escores} we deduce that the two converse implications hold and this concludes the proof. The assertion in parenthesis follows from this same lemma.
\end{proof}

We deduce the same kind of consequence as in level one (cf. Corollary~\ref{corollary:GA_cores_bij}). Note that the proof is a bit different.

\begin{corollary}
The map $\Set \to \{(e,\mcharge)$-cores$\}$ given by $w \mapsto w\cdot \mempty$ is a bijection.
\end{corollary}

\begin{proof}
Let us first prove that the action of $\Set$ on the set of $(e,\mcharge)$-cores  is transitive: let $\mpart$ be an $(e,\mcharge)$-core and let us prove that there exists $w \in \Set$ such that $\mpart = w\cdot \mempty$.   We proceed by induction on the number of boxes of $\mpart$. If $\mpart = \mempty$ then $\mpart = 1\cdot\mempty$ and we are done. Now if~$\mpart$ has at least one box, this box is a removable $i$-node for some $i \in \{0,\dots,e-1\}$ and thus $\mpart'\coloneqq s_i\mpart$ has strictly less boxes than $\mpart$. Hence, by induction there exists $w' \in \Set$ such that $\mpart' = w' \cdot \mempty$ thus $\mpart = s_i\mpart = s_iw' \cdot \mempty$, as desired.

Let us now prove that the stabiliser of $\mempty$ under the action of $\Set$ is trivial: let  $w\in\Set\setminus\{1\}$, let $\mpart \coloneqq w\cdot\mempty$ by the associated $(e,\mcharge)$-core and let us prove that $\mpart\neq \mempty$.
If $w = s_{i_1}\cdots s_{i_n}$ is a reduced expression, then $\ell(s_{i_1}w) < \ell(w)$ thus by Proposition~\ref{proposition:length_addable} we know that $\mpart$ has a removable $i_1$-node. Thus $\mpart \neq \mempty$ and this concludes the proof.
\end{proof}

Together with Corollary~\ref{corollary:lascoux_whole}, we obtain  the following result.

\begin{proposition}
\label{proposition:lattices}
Recall that $\mcharge = (0,1,\dots,e-1)$. 
The lattice $\Set$ for the Bruhat order is isomorphic to the lattice of $(e,\mcharge)$-cores for the order given by the inclusions of multi-Young diagrams.
\end{proposition}

\end{document}